\documentclass[envcountsame,envcountsect]{svmult}

\usepackage{amsfonts,amsmath,amssymb}
\usepackage{txfonts}
\usepackage{helvet}         
\usepackage{courier}        
\usepackage[bottom]{footmisc}
\usepackage{tikz}

\spnewtheorem{conj}{Conjecture}{\bfseries}{\rmfamily}

\smartqed

\begin{document}

\title*{Structured Random Matrices}

\author{Ramon van Handel}

\institute{
\textit{Contribution to IMA volume ``Discrete Structures: Analysis and 
Applications'' (2016), Springer.}\newline
Sherrerd Hall 227, Princeton University,
Princeton, NJ 08544.
\email{rvan@princeton.edu}}

\maketitle

\vskip-3cm

\abstract{Random matrix theory is a well-developed area of probability 
theory that has numerous connections with other areas of 
mathematics and its applications. Much of the literature in this area is 
concerned with matrices that possess many exact or approximate symmetries, 
such as matrices with i.i.d.\ entries, for which precise analytic results 
and limit theorems are available. Much less well understood are matrices 
that are endowed with an arbitrary structure, such as sparse Wigner 
matrices or matrices whose entries possess a given variance pattern. The 
challenge in investigating such structured random matrices is to 
understand how the given structure of the matrix is reflected in its 
spectral properties.  This chapter reviews a number of recent results, 
methods, and open problems in this direction, with a particular emphasis on
sharp spectral norm inequalities for Gaussian random matrices.}

\section{Introduction}

The study of random matrices has a long history in probability, 
statistics, and mathematical physics, and continues to be a source of 
many interesting old and new mathematical problems \cite{AGZ10,Tao12}.  
Recent years have seen impressive advances in this area, particularly in 
the understanding of universality phenomena that are exhibited by the 
spectra of classical random matrix models \cite{EY12,TV14}. At the same 
time, random matrices have proved to be of major importance in 
contemporary applied mathematics, see, for example, \cite{Tro15,Ver12} 
and the references therein.

Much of classical random matrix theory is concerned with highly 
symmetric models of random matrices. For example, the simplest random 
matrix model, the \emph{Wigner matrix}, is a symmetric matrix whose 
entries above the diagonal are independent and identically distributed. 
If the entries are chosen to be Gaussian (and the diagonal entries are 
chosen to have the appropriate variance), this model is additionally 
invariant under orthogonal transformations. Such strong symmetry 
properties make it possible to obtain extremely precise analytic results 
on the asymptotic properties of macroscopic and microscopic spectral 
statistics of these matrices, and give rise to deep connections with 
classical analysis, representation theory, combinatorics, and various 
other areas of mathematics \cite{AGZ10,Tao12}.

Much less is understood, however, once we depart from such highly 
symmetric settings and introduce nontrivial structure into the random 
matrix model. Such models are the topic of this chapter. To illustrate 
what we mean by ``structure,'' let us describe some typical examples 
that will be investigated in the sequel.
\begin{itemize}
\item 
A \emph{sparse Wigner matrix} is a matrix with a given (deterministic) 
sparsity pattern, whose nonzero entries above the diagonal are i.i.d.\ 
centered random variables. Such models have interesting applications in 
combinatorics and computer science (see, for example, \cite{AKM13}), and 
specific examples such as random band matrices are of significant 
interest in mathematical physics (cf.\ \cite{Sod10}). The ``structure'' 
of the matrix is determined by its sparsity pattern. We would like
to know how the given sparsity pattern is reflected in the spectral
properties of the matrix.
\vskip.2cm
\item Let $x_1,\ldots,x_s$ be deterministic vectors. Matrices of the form
$$
	X = \sum_{k=1}^s g_k x_kx_k^*,
$$
where $g_1,\ldots,g_s$ are i.i.d.\ standard Gaussian random variables, arise
in functional analysis (see, for example, \cite{Rud99}). The ``structure''
of the matrix is determined by the positions of the vectors $x_1,\ldots,x_s$.
We would like to know how the given positions are reflected in the spectral
properties of the matrix.
\vskip.2cm
\item Let $X_1,\ldots,X_n$ be i.i.d.\ random vectors with
covariance matrix $\Sigma$. Consider
$$
	Z = \frac{1}{n}\sum_{k=1}^n X_kX_k^*,
$$
the \emph{sample covariance matrix} \cite{Ver12,KL16}.
If we think of $X_1,\ldots,X_n$
are observed data from an underlying distribution, we
can think of $Z$ as an unbiased estimator of the covariance matrix
$\Sigma=\mathbf{E}Z$. The ``structure''
of the matrix is determined by the covariance matrix $\Sigma$.
We would like to know how the given covariance matrix is reflected 
in the spectral properties of $Z$ (and particularly in $\|Z-\Sigma\|$).
\end{itemize}
While these models possess distinct features, we will refer to such models
collectively as \emph{structured random matrices}. We emphasize two important
features of such models. First, the symmetry properties that characterize
classical random matrix models are manifestly absent in the structured
setting. Second, it is evident in the above models that it does not make
much sense to investigate their asymptotic properties (that is, probabilistic
limit theorems): as the structure is defined for the given matrix only, there
is no natural way to take the size of these matrices to infinity.

Due to these observations, the study of structured random matrices 
has a significantly different flavor than most of classical random 
matrix theory. In the absence of asymptotic theory, our main interest is 
to obtain nonasymptotic \emph{inequalities} that identify what 
structural parameters control the macroscopic properties of the 
underlying random matrix. In this sense, the study of structured random 
matrices is very much in the spirit of probability in Banach spaces 
\cite{LT91}, which is heavily reflected in the type of results that have 
been obtained in this area. In particular, the aspect of structured 
random matrices that is most well understood is the behavior of matrix 
norms, and particularly the spectral norm, of such matrices. The 
investigation of the latter will be the focus of the remainder of this 
chapter.

In view of the above discussion, it should come as no surprise that some 
of the earliest general results on structured random matrices appeared 
in the functional analysis literature \cite{TJ74,LP86,Lat05}, but 
further progress has long remained relatively limited. More recently, 
the study of structured random matrices has received renewed attention 
due to the needs of applied mathematics, cf.\ \cite{Tro15} and the 
references therein.  However, significant new progress was made in the 
past few years. On the one hand, suprisingly sharp inequalities were 
recently obtained for certain random matrix models, particularly in the 
case of independent entries, that yield nearly optimal bounds and go 
significantly beyond earlier results. On the other hand, very simple new 
proofs have been discovered for some (previously) deep classical results 
that shed new light on the underlying mechanisms and that point the way 
to further progress in this direction. The opportunity therefore seems 
ripe for an elementary presentation of the results in this area. The 
present chapter represents the author's attempt at presenting some of 
these ideas in a cohesive manner.

Due to the limited capacity of space and time, it is certainly impossible 
to provide an encylopedic presentation of the topic of this chapter, and
some choices had to be made. In particular, the following focus is adopted
throughout this chapter:
\begin{itemize}
\item The emphasis throughout is on spectral norm inequalities for
\emph{Gaussian} random matrices. The reason for this is twofold.
On the one hand, much of the difficulty of capturing the structure
of random matrices arises already in the Gaussian setting, so that
this provides a particularly clean and rich playground for investigating
such problems. On the other hand, Gaussian results extend readily to much
more general distributions, as will be discussed further in section 
\ref{sec:seginer}.
\vskip.2cm
\item For simplicity of presentation,
no attempt was made to optimize the universal constants that appear
in most of our inequalities, even though many of these inequalities can
in fact be obtained with surprisingly sharp (even optimal) constants.
The original references can be consulted for more precise statements.
\vskip.2cm
\item The presentation is by no means exhaustive, and many variations on
and extensions of the presented material have been omitted.
None of the results in this chapter are original, though I have done my
best to streamline the presentation. On the other hand, I have tried to make
the chapter as self-contained as possible, and most results are presented
with complete proofs.
\end{itemize}

The remainder of this chapter is organized as follows. The preliminary 
section \ref{sec:how} sets the stage by discussing the basic methods that 
will be used throughout this chapter to bound spectral norms of random 
matrices. Section \ref{sec:nck} is devoted to a family of powerful but 
suboptimal inequalities, the noncommutative Khintchine inequalities, that 
are applicable to the most general class of structured random matrices 
that we will encounter. In section \ref{sec:indep}, we specialize to 
structured random matrices with independent entries (such as sparse Wigner 
matrices) and derive nearly optimal bounds. We also discuss a few 
fundamental open problems in this setting. We conclude this chapter in the 
short section \ref{sec:sampcov} by investigating sample covariance 
matrices.

\section{How to bound matrix norms}
\label{sec:how}

As was discussed in the introduction, the investigation of random matrices 
with arbitrary structure has by its nature a nonasymptotic flavor: we aim 
to obtain probabilistic inequalities (upper and lower bounds) on spectral 
properties of the matrices in question that capture faithfully the 
underlying structure. At present, this program is largely developed in the 
setting of spectral norms of random matrices, which will be our focus 
thorughout this chapter. For completeness, we define:

\begin{definition}
The \emph{spectral norm} $\|X\|$ is the largest singular value of the 
matrix $X$.
\end{definition}

For convenience, we generally work with symmetric random matrices $X=X^*$. 
There is no loss of generality in doing so, as will be explained below.

Before we can obtain any meaningful bounds, we must first discuss some 
basic approaches for bounding the spectral norms of random matrices. The 
most important methods that are used for this purpose are collected in 
this section.

\subsection{The moment method}
\label{sec:mommeth}

Let $X$ be an $n\times n$ symmetric random matrix. The first difficulty 
one encounters in bounding the spectral norm $\|X\|$ is that the map 
$X\mapsto\|X\|$ is highly nonlinear. It is therefore not obvious how to 
efficiently relate the distribution of $\|X\|$ to the distribution of the 
entries $X_{ij}$. One of the most effective approaches to simplifying this 
relationship is obtained by applying the following elementary observation.

\begin{lemma}
\label{lem:mommeth}
Let $X$ be an $n\times n$ symmetric matrix. Then
$$
	\|X\| \asymp \mathrm{Tr}[X^{2p}]^{1/2p}
	\quad\mbox{for }p\asymp\log n.
$$
\end{lemma}

The beauty of this observation is that unlike $\|X\|$, which is a very 
complicated function of the entries of $X$, the quantity 
$\mathrm{Tr}[X^{2p}]$ is a \emph{polynomial} in the matrix entries. This 
means that $\mathbf{E}[\mathrm{Tr}[X^{2p}]]$, the $2p$-th
moment of the matrix $X$, can be evaluated explicitly and subjected to 
further analysis.  As Lemma 
\ref{lem:mommeth} implies that
$$
	\mathbf{E}[\|X\|^{2p}]^{1/2p} \asymp
	\mathbf{E}[\mathrm{Tr}[X^{2p}]]^{1/2p}
	\quad\mbox{for }p\asymp\log n,
$$
this provides a direct route to controlling the spectral norm of a 
random matrix. Various incarnations of this idea
are referred to as the \emph{moment method}.

Lemma \ref{lem:mommeth} actually has nothing to do with matrices. Given 
$x\in\mathbb{R}^n$, everyone knows that $\|x\|_p\to\|x\|_\infty$ as 
$p\to\infty$, so that $\|x\|_p\approx\|x\|_\infty$ when $p$ is large. How 
large should $p$ be for this to be the case?  The following lemma provides 
the answer.

\begin{lemma}
\label{lem:lplinf}
If $p\asymp\log n$, then $\|x\|_p\asymp\|x\|_\infty$ for all 
$x\in\mathbb{R}^n$.
\end{lemma}

\begin{proof}
It is trivial that
$$
	\max_{i\le n}|x_i|^p \le
	\sum_{i\le n}|x_i|^p \le n\max_{i\le n}|x_i|^p.
$$
Thus $\|x\|_\infty \le \|x\|_p \le n^{1/p}\|x\|_\infty$, and
$n^{1/p} = e^{(\log n)/p} \asymp 1$ when $\log n\asymp p$.
\qed
\end{proof}

The proof of Lemma \ref{lem:mommeth} follows readily by applying
Lemma \ref{lem:lplinf} to the spectrum.

\begin{proof}[Proof of Lemma \ref{lem:mommeth}]
Let $\lambda=(\lambda_1,\ldots,\lambda_n)$ be the eigenvalues of $X$.
Then $\|X\|=\|\lambda\|_\infty$ and $\mathrm{Tr}[X^{2p}]^{1/2p} =
\|\lambda\|_{2p}$. The result follows from Lemma \ref{lem:lplinf}.
\qed
\end{proof}

The moment method will be used frequently throughout this chapter as the 
first step in bounding the spectral norm of random matrices.  However, the 
moment method is just as useful in the vector setting. As a warmup 
exercise, let us use this approach to bound the maximum of i.i.d.\ 
Gaussian random variables (which can be viewed as a vector analogue of 
bounding the maximum eigenvalue of a random matrix). If $g\sim N(0,I)$ 
is the standard Gaussian vector in $\mathbb{R}^n$, Lemma \ref{lem:lplinf} 
implies
$$
	[\mathbf{E}\|g\|_\infty^p]^{1/p} \asymp
	[\mathbf{E}\|g\|_p^p]^{1/p} \asymp
	[\mathbf{E}g_1^p]^{1/p}
	\quad\mbox{for }p\asymp\log n.
$$
Thus the problem of bounding the maximum of $n$ i.i.d.\ Gaussian random 
variables is reduced by the moment method to computing the
$\log n$-th moment of a single Gaussian random variable.
We will bound the latter in section \ref{sec:nckh} in preparation 
for proving the analogous bound for random matrices.
For our present purposes, let us simply note the outcome of this 
computation $[\mathbf{E}g_1^p]^{1/p}\lesssim\sqrt{p}$
(Lemma \ref{lem:gaussmom}), so that
$$
	\mathbf{E}\|g\|_\infty 
	\le [\mathbf{E}\|g\|_\infty^{\log n}]^{1/\log n}
	\lesssim \sqrt{\log n}.
$$
This bound is in fact sharp (up to the universal constant).

\begin{remark}
\label{rem:badmoment}
Lemma \ref{lem:mommeth} implies immediately that
$$
	\mathbf{E}\|X\| \asymp \mathbf{E}[\mathrm{Tr}[X^{2p}]^{1/2p}]
	\quad\mbox{for }p\asymp\log n.
$$
Unfortunately, while this bound is sharp by construction, it is
essentially useless: 
the expectation of $\mathrm{Tr}[X^{2p}]^{1/2p}$ is in principle just as 
difficult to 
compute as that of $\|X\|$ itself.  The utility of the moment method stems 
from the fact that we can explicitly compute the expectation of 
$\mathrm{Tr}[X^{2p}]$, a polynomial in the matrix entries. This suggests 
that the moment method is well-suited in principle only for obtaining 
sharp bounds on the $p$th moment of the spectral norm
$$
	\mathbf{E}[\|X\|^{2p}]^{1/2p} \asymp
	\mathbf{E}[\mathrm{Tr}[X^{2p}]]^{1/2p}
	\quad\mbox{for }p\asymp\log n,
$$
and not on the first moment $\mathbf{E}\|X\|$ of the spectral norm.  Of 
course, as $\mathbf{E}\|X\|\le [\mathbf{E}\|X\|^{2p}]^{1/2p}$ by Jensen's 
inequality, this yields an \emph{upper} bound on the first moment of the 
spectral norm.  We will see in the sequel that this upper bound is often, 
but not always, sharp.  We can expect the moment method to yield a
sharp bound on $\mathbf{E}\|X\|$ when the fluctuations of $\|X\|$ are of a 
smaller order than its mean; this was the case, for example, in the 
computation of $\mathbf{E}\|g\|_\infty$ above.  On the other hand, the 
moment method 
is inherently dimension-dependent (as one must choose $p\sim\log n$), so 
that it is generally not well suited for obtaining dimension-free bounds.
\end{remark}

We have formulated Lemma \ref{lem:mommeth} for symmetric matrices.
A completely analogous approach can be applied to non-symmetric matrices.
In this case, we use that
$$
	\|X\|^2 = \|X^*X\| \asymp
	\mathrm{Tr}[(X^*X)^p]^{1/p}\quad
	\mbox{for }p\asymp\log n,
$$
which follows directly from Lemma \ref{lem:mommeth}. However, this 
non-symmetric form is often somewhat inconvenient in the proofs of 
random matrix bounds, or at least requires additional bookkeeping.
Instead, we recall a classical trick that allows us to directly obtain 
results for non-symmetric matrices from the analogous symmetric results.
If $X$ is any 
$n\times m$ rectangular matrix, then it is readily verified that $\|\tilde 
X\|=\|X\|$, where $\tilde X$ denotes the $(n+m)\times(n+m)$ symmetric matrix
defined by
$$
	\tilde X =
	\left[
	\begin{matrix}
	0 & X \\ X^* & 0
	\end{matrix}
	\right].
$$
Therefore, to obtain a bound on the norm $\|X\|$ of a non-symmetric random 
matrix, it suffices to apply the corresponding result for symmetric random 
matrices to the doubled matrix $\tilde X$. For this reason, it is not 
really necessary to treat non-symmetric matrices separately, and we will 
conveniently restrict our attention to symmetric matrices throughout this 
chapter without any loss of generality.

\begin{remark}
\label{rem:mxconc}
A variant on the moment method is to use the bounds
$$
	e^{t\lambda_{\rm max}(X)} \le \mathrm{Tr}[e^{tX}] \le
	n e^{t\lambda_{\rm max}(X)},
$$
which gives rise to so-called ``matrix concentration'' inequalities. 
This approach has become popular in recent years (particularly in the 
applied mathematics literature) as it provides easy proofs of a number 
of useful inequalities. Matrix concentration bounds are often stated in 
terms of tail probabilities $\mathbf{P}[\lambda_{\rm max}(X)>t]$, and 
therefore appear at first sight to provide more information than 
expected norm bounds. This is not the case, however: the resulting tail 
bounds are highly suboptimal, and much sharper inequalities can be 
obtained by combining expected norm bounds with concentration 
inequalities \cite{BLM13} or chaining tail bounds \cite{Dir15}. As in 
the case of classical concentration inequalities, the moment method 
essentially subsumes the matrix concentration approach and is often more 
powerful. We therefore do not discuss this approach further, but refer
to \cite{Tro15} for a systematic development. 
\end{remark}

\subsection{The random process method}
\label{sec:randpr}

While the moment method introduced in the previous section is very 
powerful, it has a number of drawbacks. First, while the matrix moments 
$\mathbf{E}[\mathrm{Tr}[X^{2p}]]$ can typically be computed explicitly, 
extracting useful information from the resulting expressions is a 
nontrivial matter that can result in difficult combinatorial problems. 
Moreover, as discussed in Remark \ref{rem:badmoment}, in certain cases the 
moment method \emph{cannot} yield sharp bounds on the expected spectral 
norm $\mathbf{E}\|X\|$.  Finally, the moment method can only yield 
information on the spectral norm of the matrix; if other operator norms 
are of interest, this approach is powerless.  In this section, we develop 
an entirely different method that provides a fruitful approach for 
addressing these issues.

The present method is based on the following trivial fact.

\begin{lemma}
\label{lem:rptriv}
Let $X$ be an $n\times n$ symmetric matrix. Then
$$
	\|X\| = \sup_{v\in B}|\langle v,Xv\rangle|,
$$
where $B$ denotes the Euclidean unit ball in $\mathbb{R}^n$.
\end{lemma}

When $X$ is a symmetric random matrix, we can view $v\mapsto\langle 
v,Xv\rangle$ as a \emph{random process} that is indexed by the Euclidean 
unit ball.  Thus controlling the expected spectral norm of $X$ is 
none other than a special instance of the general probabilistic problem of 
controlling the expected supremum of a random process. There exist a 
powerful methods for this purpose (see, e.g., \cite{Tal14}) that could 
potentially be applied in the present setting to generate insight on the 
structure of random matrices.

Already the simplest possible approach to bounding the suprema of random 
processes, the $\varepsilon$-net method, has proved to be very useful in 
the study of basic random matrix models. The idea behind this approach is 
to approximate the supremum over the unit ball $B$ by the maximum over a 
finite discretization $B_\varepsilon$ of the unit ball, which reduces the 
problem to computing the maximum of a finite number of random variables 
(as we did, for example, in the previous section when we computed 
$\|g\|_\infty$). Let us briefly sketch how this approach works in the 
following basic example.
Let $X$ be the $n\times n$ symmetric random matrix with i.i.d.\ 
standard Gaussian entries above the diagonal. Such a matrix is called a 
\emph{Wigner matrix}. Then for every vector $v\in B$, the random variable 
$\langle v,Xv\rangle$ is Gaussian with variance at most $2$. Now let 
$B_{\varepsilon}$ be a finite subset of the unit ball $B$ in 
$\mathbb{R}^n$ such that every point in $B$ is within distance at most 
$\varepsilon$ from a point in $B_\varepsilon$. Such a set is called an 
\emph{$\varepsilon$-net}, and should be viewed as a uniform discretization 
of the unit ball $B$ at the scale $\varepsilon$.  Then we can bound, for 
small $\varepsilon$,\footnote{
	The first inequality follows by noting that for every $v\in B$,
	choosing $\tilde v\in B_\varepsilon$ such that $\|v-\tilde 
	v\|\le\varepsilon$, we have
	$|\langle v,Xv\rangle| =
	|\langle\tilde v,X\tilde v\rangle 
	+ \langle v-\tilde v,X(v+\tilde v)\rangle|
	\le
	|\langle\tilde v,X\tilde v\rangle| + 2\varepsilon\|X\|$.
}
$$
	\mathbf{E}\|X\| = 
	\mathbf{E}\sup_{v\in B}|\langle v,Xv\rangle|
	\lesssim
	\mathbf{E}\sup_{v\in B_{\varepsilon}}|\langle v,Xv\rangle|
	\lesssim
	\sqrt{\log |B_\varepsilon|},
$$
where we used that the expected maximum of $k$ Gaussian random variables 
with variance $\lesssim 1$ is bounded by $\lesssim \sqrt{\log k}$ (we 
proved this in the previous section using the moment method: note that 
independence was not needed for the upper bound.) A classical argument 
shows that the smallest $\varepsilon$-net in $B$ has cardinality of order 
$\varepsilon^{-n}$, so the above argument yields a bound of order 
$\mathbf{E}\|X\|\lesssim\sqrt{n}$ for Wigner matrices. 
It turns out that this bound is in fact sharp in the present setting: 
Wigner matrices satisfy $\mathbf{E}\|X\|\asymp\sqrt{n}$ (we will prove 
this more carefully in section \ref{sec:fail} below).

Variants of the above argument have proved to be very useful in random 
matrix theory, and we refer to \cite{Ver12} for a systematic development. 
However, $\varepsilon$-net arguments are usually applied to highly 
symmetric situations, such as is the case for Wigner matrices (all entries 
are identically distributed). The problem with the $\varepsilon$-net 
method is that it is sharp essentially only in this situation: this method 
cannot incorporate nontrivial structure. To illustrate this, consider the 
following typical structured example. Fix a certain sparsity pattern of 
the matrix $X$ at the outset (that is, choose a subset of the entries that 
will be forced to zero), and choose the remaining entries to be 
independent standard Gaussians. In this case, a ``good'' discretization of 
the problem cannot simply distribute points uniformly over the unit ball 
$B$, but rather must take into account the geometry of the given sparsity 
pattern. Unfortunately, it is entirely unclear how this is to be 
accomplished in general. For this reason, $\varepsilon$-net methods have 
proved to be of limited use for \emph{structured} random matrices, and 
they will play essentially no role in the remainder of this chapter.

\begin{remark}
\label{rem:gench}
Deep results from the theory of Gaussian processes \cite{Tal14} guarantee 
that the expected supremum of any Gaussian process and of many other 
random processes can be captured sharply by a sophisticated multiscale 
counterpart of the $\varepsilon$-net method called the generic chaining.
Therefore, in principle, it should be possible to capture precisely the 
norm of structured random matrices if one is able to construct a 
near-optimal multiscale net. Unfortunately, the general theory only 
guarantees the existence of such a net, and provides essentially no 
mechanism to construct one in any given situation. From this 
perspective, structured random matrices provide a particularly interesting 
case study of inhomogeneous random processes whose investigation could 
shed new light on these more general mechanisms (this perspective provided 
strong motivation for this author's interest in random matrices). At 
present, however, progress along these lines remains in a very primitive 
state. Note that even the most trivial of examples from the random 
matrix perspective, such as the case where $X$ is a diagonal matrix with 
i.i.d.\ Gaussian entries on the diagonal, require already a delicate 
multiscale net to obtain sharp results; see, e.g., \cite{vH16}.
\end{remark}

As direct control of the random processes that arise from structured 
random matrices is largely intractable, a different approach is needed. To 
this end, the key idea that we will exploit is the use of 
\emph{comparison theorems} to bound the expected supremum of one random 
process by that of another random process. The basic idea is to design a 
suitable comparison process that dominates the random process of Lemma 
\ref{lem:rptriv} but that is easier to control. For this approach to be 
successful, the comparison process must capture the structure of the 
original process while at the same time being amenable to some form of 
explicit computation. In principle there is no reason to expect that this 
is ever possible. Nonetheless, we will repeatedly apply different 
variations on this approach to obtain the best known bounds on structured 
random matrices. Comparison methods are a recurring theme throughout this 
chapter, and we postpone further discussion to the following sections.

Let us note that the random process method is easily extended also to 
non-symmetric matrices: if $X$ is an $n\times m$ rectangular matrix, we 
have
$$
	\|X\| = \sup_{v,w\in B}\langle v,Xw\rangle.
$$
Alternatively, we can use the same symmetrization trick as was illustrated 
in the previous section to reduce to the symmetric case. For this reason, 
we will restrict attention to symmetric matrices in the sequel. Let us 
also note, however, that unlike the moment method, the present approach 
extends readily to other operator norms by replacing the Euclidean unit 
ball $B$ by the unit ball for other norms. In this sense, the random 
process method is substantially more general than the moment method, which 
is restricted to the spectral norm. However, the spectral norm is often 
the most interesting norm in practice in applications of random matrix theory.

\subsection{Roots and poles}

The moment method and random process method discussed in the previous 
sections have proved to be by far the most useful approaches to bounding 
the spectral norms of random matrices, and all results in this chapter 
will be based on one or both of these methods. We want to briefly mention 
a third approach, however, that has recently proved to be useful. It is 
well-known from linear algebra that the eigenvalues of a symmetric matrix 
$X$ are the roots of the characteristic polynomial
$$
	\chi(t) = \det(tI-X),
$$
or, equivalently, the poles of the Stieltjes transform
$$
	s(t) := \mathrm{Tr}[(tI-X)^{-1}] =
	\frac{d}{dt}\log\chi(t).
$$
One could therefore attempt to bound the extreme eigenvalues of $X$ (and 
therefore the spectral norm $\|X\|$) by controlling the location of the 
largest root (pole) of the characteristic polynomial (Stieltjes tranform) 
of $X$, with high probability.

The Stieltjes transform method plays a major role in random matrix theory 
\cite{AGZ10}, as it provides perhaps the simplest route to proving limit 
theorems for the spectral distributions of random matrices. It is possible 
along these lines to prove asymptotic results on the extreme eigenvalues, 
see \cite{BS98} for example. However, as the Stieltjes transform is highly 
nonlinear, it seems to be very difficult to use this approach to address 
nonasymptotic questions for structured random matrices where explicit 
limit information is meaningless. The characteristic polynomial appears at 
first sight to be more promising, as this is a polynomial in the matrix 
entries: one can therefore hope to compute $\mathbf{E}\chi$ exactly. This 
simplicity is deceptive, however, as there is no reason to expect that 
$\mathrm{maxroot}(\mathbf{E}\chi)$ has any relation to the quantity 
$\mathbf{E}\,\mathrm{maxroot}(\chi)$ that we are interested in. It was 
therefore long believed that such an approach does not provide any useful 
tool in random matrix theory. Nonetheless, a determinisitic version of 
this idea plays the crucial role in the recent breakthrough resolution of 
the Kadison-Singer conjecture \cite{MSS15}, so that it is conceivable that 
such an approach could prove to be fruitful in problems of random matrix 
theory (cf.\ \cite{SV13} where related ideas were applied to Stieltjes 
transforms in a random matrix problem). To date, however, these methods 
have not been successefully applied to the problems investigated in this 
chapter, and they will make no further appearance in the sequel.

\section{Khintchine-type inequalities}
\label{sec:nck}

The main aim of this section is to introduce a very general method for 
bounding the spectral norm of structured random matrices. The basic idea, 
due to Lust-Piquard \cite{LP86}, is to prove an analog of the classical 
Khintchine inequality for scalar random variables in the noncommutative 
setting.  This \emph{noncommutative Khintchine inequality} allows us to 
bound the moments of structured random matrices, which immediately results 
in a bound on the spectral norm by Lemma \ref{lem:mommeth}.

The advantage of the noncommutative Khintchine inequality is that it can 
be applied in a remarkably general setting: it does not even require 
independence of the matrix entries. The downside of this inequality is 
that it almost always gives rise to bounds on the spectral norm that are 
suboptimal by a multiplicative factor that is logarithmic in the dimension 
(cf.\ section \ref{sec:bvh}). We will discuss the origin of this 
suboptimality and some potential methods for reducing it in the general 
setting of this section. Much sharper bounds will be obtained in section 
\ref{sec:indep} under the additional restriction that the matrix entries 
are independent.

For simplicity, we will restrict our attention to matrices with Gaussian 
entries, though extensions to other distributions are easily obtained 
(for example, see \cite{MJCFT14}).

\subsection{The noncommutative Khintchine inequality}
\label{sec:nckh}

In this section, we will consider the following very general setting. Let 
$X$ be an $n\times n$ symmetric random matrix with zero mean. The only 
assumption we make on the distribution is that the entries on and above 
the diagonal (that is, those entries that are not fixed by symmetry) are 
centered and jointly Gaussian. In particular, these entries can possess an 
arbitrary covariance matrix, and are assumed to be neither identically 
distributed nor independent. Our aim is to bound the spectral norm $\|X\|$ 
in terms of the given covariance structure of the matrix.

It proves to be convenient to reformulate our random matrix model 
somewhat. Let $A_1,\ldots,A_s$ be \emph{nonrandom} $n\times n$ symmetric 
matrices, and let $g_1,\ldots,g_s$ be independent standard Gaussian 
variables. Then we define the matrix $X$ as
$$
	X = \sum_{k=1}^s g_kA_k.
$$
Clearly $X$ is a symmetric matrix with jointly Gaussian entries. 
Conversely, the reader will convince herself after a moment's reflection 
that any symmetric matrix with centered and jointly Gaussian entries can 
be written in the above form for some choice of $s\le n(n+1)/2$ and 
$A_1,\ldots,A_s$. There is therefore no loss of generality in considering 
the present formulation (we will reformulate our ultimate bounds in a way 
that does not depend on the choice of the coefficient matrices $A_k$).

Our intention is to apply the moment method. To this end, we must obtain 
bounds on the moments $\mathbf{E}[\mathrm{Tr}[X^{2p}]]$ of the matrix $X$. 
It is instructive to begin by considering the simplest possible case where 
the dimension $n=1$.  In this case, $X$ is simply a scalar Gaussian random 
variable with zero mean and variance $\sum_kA_k^2$, and the problem in 
this case reduces to bounding the moments of a scalar Gaussian variable.

\begin{lemma}
\label{lem:gaussmom}
Let $g\sim N(0,1)$. Then $\mathbf{E}[g^{2p}]^{1/2p}\le\sqrt{2p-1}$.
\end{lemma}

\begin{proof}
We use the following fundamental \emph{gaussian integration by parts} 
property:
$$
	\mathbf{E}[gf(g)] = \mathbf{E}[f'(g)].
$$
To prove it, simply note that integration by parts yields
$$
	\int_{-\infty}^\infty x f(x)\,\frac{e^{-x^2/2}}{\sqrt{2\pi}}\,dx
	=
	\int_{-\infty}^\infty 
	\frac{df(x)}{dx}\,\frac{e^{-x^2/2}}{\sqrt{2\pi}}\,dx
$$
for smooth functions $f$ with compact support, and the conclusion is 
readily extended by approximation to any $C^1$ function for which the 
formula makes sense.

We now apply the integration by parts formula to $f(x)=x^{2p-1}$ as
follows:
$$
	\mathbf{E}[g^{2p}] = \mathbf{E}[g\cdot g^{2p-1}] =
	(2p-1)\mathbf{E}[g^{2p-2}] \le
	(2p-1)\mathbf{E}[g^{2p}]^{1-1/p},
$$
where the last inequality is by Jensen. Rearranging yields the conclusion.
\qed
\end{proof}

Applying Lemma \ref{lem:gaussmom} yields immediately that
$$
	\mathbf{E}[X^{2p}]^{1/2p}\le
	\sqrt{2p-1}\,\Bigg[\sum_{k=1}^s A_k^2\Bigg]^{1/2}
	\quad\mbox{when }n=1.
$$
It was realized by Lust-Piquard \cite{LP86} that the 
analogous inequality holds in any dimension $n$ (the correct dependence of 
the bound on $p$ was obtained later, cf.\ \cite{Pis03}).

\begin{theorem}[Noncommutative Khintchine inequality]
\label{thm:nck}
In the present setting
$$
	\mathbf{E}[\mathrm{Tr}[X^{2p}]]^{1/2p} \le
	\sqrt{2p-1}\,\mathrm{Tr}\Bigg[
	\Bigg(\sum_{k=1}^s A_k^2\Bigg)^p\Bigg]^{1/2p}.
$$
\end{theorem}

By combining this bound with Lemma \ref{lem:mommeth}, we immediately 
obtain the following conclusion regarding the spectral norm of the matrix 
$X$.

\begin{corollary}
\label{cor:nck}
In the setting of this section,
$$
	\mathbf{E}\|X\| \lesssim \sqrt{\log n}\,
	\Bigg\|\sum_{k=1}^sA_k^2\Bigg\|^{1/2}.
$$
\end{corollary}

This bound is expressed directly in terms of the coefficient matrices 
$A_k$ that determine the structure of $X$, and has proved to be extremely 
useful in applications of random matrix theory in functional analysis and 
applied mathematics.  To what extent this bound is sharp will be discussed 
in the next section.

\begin{remark}
Recall that our bounds apply to any symmetric matrix $X$ with centered and 
jointly Gaussian entries. Our bounds should therefore not depend on the 
choice of representation in terms of the coefficient 
matrices $A_k$, which is not unique.  It is easily verified that this is 
the case. Indeed, it suffices to note that
$$
	\mathbf{E}X^2 = \sum_{k=1}^s A_k^2,
$$
so that we can express the conclusion of Theorem \ref{thm:nck} 
and Corollary \ref{cor:nck} as
$$
	\mathbf{E}[\mathrm{Tr}[X^{2p}]]^{1/2p} \lesssim \sqrt{p}\,
	\mathrm{Tr}[(\mathbf{E}X^2)^p]^{1/2p},\qquad\quad
	\mathbf{E}\|X\| \lesssim \sqrt{\log n}\,
	\|\mathbf{E}X^2\|^{1/2}
$$
without reference to the coefficient matrices $A_k$. We note that the 
quantity $\|\mathbf{E}X^2\|$ has a natural interpretation: it measures
the size of the matrix $X$ ``on average'' (as the expectation in this
quantity is \emph{inside} the spectral norm).
\end{remark}

We now turn to the proof of Theorem \ref{thm:nck}.  We begin by noting 
that the proof follows immediately from Lemma \ref{lem:gaussmom} not just 
when $n=1$, but also in any dimension $n$ under the additional assumption 
that the matrices $A_1,\ldots,A_s$ commute.  Indeed, in this case we can 
work without loss of generality in a basis in which all the matrices $A_k$ 
are simultaneously diagonal, and the result follows by applying Lemma 
\ref{lem:gaussmom} to every diagonal entry of $X$. The crucial idea behind 
the proof of Theorem \ref{thm:nck} is that \emph{the commutative case is 
in fact the worst case situation}! This idea will appear very explicitly 
in the proof: we will simply repeat the proof of Lemma \ref{lem:gaussmom}, 
and the result will follow by showing that we can permute the order of the 
matrices $A_k$ at the pivotal point in the proof.  (The simple proof given 
here follows \cite{Tro15b}.)

\begin{proof}[Proof of Theorem \ref{thm:nck}]
As in the proof of Lemma \ref{lem:gaussmom}, we obtain
\begin{align*}
	\mathbf{E}[\mathrm{Tr}[X^{2p}]] &=
	\mathbf{E}[\mathrm{Tr}[X\cdot X^{2p-1}]] \\ &=
	\sum_{k=1}^s
	\mathbf{E}[g_k\mathrm{Tr}[A_k X^{2p-1}]] \\ &=
	\sum_{\ell=0}^{2p-2}
	\sum_{k=1}^s
	\mathbf{E}[\mathrm{Tr}[A_kX^\ell A_k X^{2p-2-\ell}]]
\end{align*}
using Gaussian integration by parts. The crucial step in the proof is the 
observation that permuting $A_k$ and $X^\ell$ inside the trace can only 
increase the bound.

\begin{lemma}
\label{lem:exchg}
$\mathrm{Tr}[A_kX^\ell A_k X^{2p-2-\ell}] \le
\mathrm{Tr}[A_k^2X^{2p-2}]$.
\end{lemma}

\begin{proof}
Let us write $X$ in terms of its eigendecomposition
$X = \sum_{i=1}^n \lambda_i v_iv_i^*$,
where $\lambda_i$ and $v_i$ denote the eigenvalues and eigenvectors of 
$X$.  Then we can write
$$
	\mathrm{Tr}[A_kX^\ell A_k X^{2p-2-\ell}] =
	\sum_{i,j=1}^n
	\lambda_i^\ell \lambda_j^{2p-2-\ell}
	|\langle v_i,A_kv_j\rangle|^2 \le
	\sum_{i,j=1}^n
	|\lambda_i|^\ell |\lambda_j|^{2p-2-\ell}
	|\langle v_i,A_kv_j\rangle|^2.
$$
But note that the right-hand side is a convex function of $\ell$, so that 
its maximum in the interval $[0,2p-2]$ is attained either at $\ell=0$ or 
$\ell=2p-2$.  This yields
$$
	\mathrm{Tr}[A_kX^\ell A_k X^{2p-2-\ell}] \le
	        \sum_{i,j=1}^n
        |\lambda_j|^{2p-2}
        |\langle v_i,A_kv_j\rangle|^2 =
	\mathrm{Tr}[A_k^2X^{2p-2}],
$$
and the proof is complete.
\qed
\end{proof}

We now complete the proof of the noncommutative Khintchine inequality.
Substituting Lemma \ref{lem:exchg} into the previous inequality yields
\begin{align*}
	\mathbf{E}[\mathrm{Tr}[X^{2p}]] &\le
	(2p-1)
	\sum_{k=1}^s\mathbf{E}[\mathrm{Tr}[A_k^2 X^{2p-2}]]
	\\ &\le
	(2p-1)\,
	\mathrm{Tr}\Bigg[
        \Bigg(\sum_{k=1}^s A_k^2\Bigg)^p\Bigg]^{1/p}
	\mathbf{E}[\mathrm{Tr}[X^{2p}]]^{1-1/p},
\end{align*}
where we used H\"older's inequality $\mathrm{Tr}[YZ] \le
\mathrm{Tr}[|Y|^p]^{1/p}\mathrm{Tr}[|Z|^{p/(p-1)}]^{1-1/p}$ in the last 
step.  Rearranging this expression yields the desired conclusion.
\qed
\end{proof}

\begin{remark}
The proof of Corollary \ref{cor:nck} given here, using the moment
method, is exceedingly simple. However, by its nature, it can only
bound the spectral norm of the matrix, and would be useless if we wanted
to bound other operator norms. It is worth noting that an alternative
proof of Corollary \ref{cor:nck} was developed by Rudelson, using deep
random process machinery described in Remark \ref{rem:gench}, for the
special case where the matrices $A_k$ are all of rank one (see
\cite[Prop.\ 16.7.4]{Tal14} for an exposition of this proof). The
advantage of this approach is that it extends to some other operator
norms, which proves to be useful in Banach space theory. It is remarkable,
however, that no random process proof of Corollary 
\ref{cor:nck} is known to date in the general setting.
\end{remark}

\subsection{How sharp are Khintchine inequalities?}
\label{sec:fail}

Corollary \ref{cor:nck} provides a very convenient bound on the spectral 
norm $\|X\|$:  it is expressed directly in terms of the coefficients $A_k$ 
that define the structure of the matrix $X$. However, is this structure 
captured \emph{correctly}?  To understand the degree to which Corollary 
\ref{cor:nck} is sharp, let us augment it with a lower bound.

\begin{lemma}
\label{lem:lowernck}
Let $X=\sum_{k=1}^s g_kA_k$ as in the previous section. Then
$$
	\Bigg\|\sum_{k=1}^sA_k^2\Bigg\|^{1/2} \lesssim
	\mathbf{E}\|X\| \lesssim \sqrt{\log n}\,
	\Bigg\|\sum_{k=1}^sA_k^2\Bigg\|^{1/2}.
$$
That is, the noncommutative Khintchine bound is sharp up to a
logarithmic factor.
\end{lemma}

\begin{proof}
The upper bound in Corollary \ref{cor:nck}, and it remains to prove the 
lower bound. A slightly simpler bound is immediate by Jensen's inequality: 
we have
$$
	\mathbf{E}\|X\|^2 \ge
	\|\mathbf{E}X^2\| =
	\Bigg\|\sum_{k=1}^sA_k^2\Bigg\|.
$$
It therefore remains to show that $(\mathbf{E}\|X\|)^2\gtrsim 
\mathbf{E}\|X\|^2$, or, equivalently, that
$\mathrm{Var}\|X\|\lesssim(\mathbf{E}\|X\|)^2$. To bound the fluctuations 
of 
the spectral norm, we recall an important property of Gaussian random 
variables (see, for example, \cite{Pis86}).

\begin{lemma}[Gaussian concentration]
\label{lem:gaussconc}
Let $g$ be a standard Gaussian vector in $\mathbb{R}^n$,
let $f:\mathbb{R}^n\to\mathbb{R}$ be smooth, and let $p\ge 1$.
Then
$$
	[\mathbf{E}(f(g)-\mathbf{E}f(g))^p]^{1/p}
	\lesssim
	\sqrt{p}\, [\mathbf{E}\|\nabla f(g)\|^p]^{1/p}.
$$
\end{lemma}

\begin{proof}
Let $g'$ be an independent copy of $g$, and define
$g(\varphi) = g\sin\varphi+g'\cos\varphi$.  Then
$$
	f(g)-f(g') =
	\int_0^{\pi/2} \frac{d}{d\varphi} f(g(\varphi))\,d\varphi
	=
	\int_0^{\pi/2} \langle g'(\varphi),\nabla f(g(\varphi))
	\rangle\,d\varphi,
$$
where $g'(\varphi)=\frac{d}{d\varphi}g(\varphi)$.
Applying Jensen's inequality twice gives
$$
	\mathbf{E}(f(g)-\mathbf{E}f(g))^p \le
	\mathbf{E}(f(g)-f(g'))^p \le
	\frac{2}{\pi}
	\int_0^{\pi/2} 
	\mathbf{E}(\tfrac{\pi}{2}\langle g'(\varphi),\nabla f(g(\varphi))
	\rangle)^p\,d\varphi.
$$
Now note that $(g(\varphi),g'(\varphi)) \stackrel{d}{=} (g,g')$ 
for every $\varphi$. We can therefore apply Lemma \ref{lem:gaussmom}
conditionally on $g(\varphi)$ to estimate for every $\varphi$
$$
	[\mathbf{E}\langle g'(\varphi),\nabla f(g(\varphi))
        \rangle^p]^{1/p} \lesssim
	\sqrt{p}\,\mathbf{E}\|\nabla f(g(\varphi))\|^p]^{1/p} =
	\sqrt{p}\,\mathbf{E}\|\nabla f(g)\|^p]^{1/p},
$$
and substituting into the above expression completes the proof.
\qed
\end{proof}

We apply Lemma \ref{lem:gaussconc} to the function
$f(x) = \|\sum_{k=1}^sx_kA_k\|$.  Note that
\begin{align*}
	|f(x)-f(x')| &\le
	\Bigg\|
	\sum_{k=1}^s(x_k-x_k')A_k
	\Bigg\| =
	\sup_{v\in B}
	\Bigg|
	\sum_{k=1}^s(x_k-x_k')\langle v,A_kv\rangle\Bigg|
	\\ &
	\le
	\|x-x'\|
	\sup_{v\in B}
	\Bigg[
	\sum_{k=1}^s\langle v,A_kv\rangle^2
	\Bigg]^{1/2}
	=: \sigma_*\|x-x'\|.
\end{align*}
Thus $f$ is $\sigma_*$-Lipschitz, so $\|\nabla f\|\le\sigma_*$, and
Lemma \ref{lem:gaussconc} yields $\mathrm{Var}\|X\|\lesssim\sigma_*^2$.
But as
$$
	\sigma_* =
	\sqrt{\frac{\pi}{2}}
	\sup_{v\in B}
	\mathbf{E}\Bigg|
	\sum_{k=1}^sg_k\langle v,A_kv\rangle
	\Bigg|
	\le
	\sqrt{\frac{\pi}{2}}\mathbf{E}\|X\|,
$$
we have $\mathrm{Var}\|X\|\lesssim(\mathbf{E}\|X\|)^2$, and
the proof is complete.
\qed
\end{proof}

Lemma \ref{lem:lowernck} shows that the structural quantity 
$\sigma:=\|\sum_{k=1}^sA_k^2\|^{1/2}=\|\mathbf{E}X^2\|^{1/2}$ that appears 
in the noncommutative 
Khintchine inequality is very natural: the expected spectral norm 
$\mathbf{E}\|X\|$ is controlled by $\sigma$ up to a logarithmic factor in 
the dimension. It is not at all clear, \emph{a priori}, whether the 
upper or lower bound in Lemma \ref{lem:lowernck} is sharp. It turns out 
that either the upper bound or the lower bound may be sharp in different 
situations. Let us illustrate this in two extreme examples.

\begin{example}[Diagonal matrix]
\label{ex:diag}
Consider the case where $X$ is a diagonal matrix
$$
	X = \left[
	\begin{matrix}
	g_1 &     &        &     \\
	    & g_2 &        &     \\
	    &     & \ddots &     \\
	    &     &        & g_n \\
	\end{matrix}
	\right]
$$
with i.i.d.\ standard Gaussian entries on the diagonal. In this case,
$$
	\mathbf{E}\|X\| = \mathbf{E}\|g\|_\infty \asymp \sqrt{\log n}.
$$
On the other hand, we clearly have
$$
	\sigma = \|\mathbf{E}X^2\|^{1/2} = 1,
$$
so the upper bound in Lemma \ref{lem:lowernck} is sharp. This shows that
the logarithmic factor in the noncommutative Khintchine inequality
\emph{cannot} be removed.
\end{example}

\begin{example}[Wigner matrix]
\label{ex:wigner}
Let $X$ be a symmetric matrix
$$
	X = \left[
        \begin{matrix}
        g_{11} & g_{12} & \cdots & g_{1n} \\
        g_{12} & g_{22} &        & g_{2n} \\
        \vdots &        & \ddots & \vdots \\
        g_{1n} & g_{2n} & \cdots & g_{nn} \\
        \end{matrix}
        \right] 
$$
with i.i.d.\ standard Gaussian entries on and above the diagonal.
In this case
$$
	\sigma = \|\mathbf{E}X^2\|^{1/2} = \sqrt{n}.
$$
Thus Lemma \ref{lem:lowernck} yields the bounds
$$
	\sqrt{n} \lesssim \mathbf{E}\|X\|\lesssim
	\sqrt{n\log n}.
$$
Which bound is sharp? A hint can be obtained from what is
perhaps the most classical result in random matrix theory: the empirical
spectral distribution of the matrix $n^{-1/2}X$ (that is, the random 
probability measure on $\mathbb{R}$ that places a point mass on every 
eigenvalue of $n^{-1/2}X$) converges weakly to the Wigner semicircle
distribution $\frac{1}{2\pi}\sqrt{(4-x^2)_+}\,dx$ \cite{AGZ10,Tao12}.
Therefore, when the dimension $n$ is large, the eigenvalues of $X$ are 
approximately distributed according to the following density:
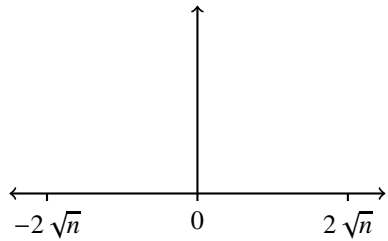
\begin{center}
\vskip.2cm
\begin{tikzpicture}
\draw[thick,<->] (0,0) to (5,0);
\draw[thick,->] (2.5,0) to (2.5,2.5);
\draw[thick] (2.5,0) to (2.5,-.1) node[below] {$0$};
\draw[thick] (.5,0) to (.5,-.1) node[below] {$-2\sqrt{n}$};
\draw[thick] (4.5,0) to (4.5,-.1) node[below] {$2\sqrt{n}$};
\draw[thick,blue] plot[samples=200,domain=.5:4.5] function {(4-(x-2.5)**2)**.5};
\end{tikzpicture}
\vskip.2cm
\end{center}
This picture strongly suggests that the spectrum of $X$ is supported
at least approximately in the interval $[-2\sqrt{n},2\sqrt{n}]$, which 
implies that $\|X\|\asymp\sqrt{n}$.  

\begin{lemma}
\label{lem:wigner}
For the Wigner matrix of Example \ref{ex:wigner},
$\mathbf{E}\|X\|\asymp\sqrt{n}$.
\end{lemma}

Thus we see that in the present example it is the \emph{lower} bound in 
Lemma \ref{lem:lowernck} that is sharp, while the upper bound obtained 
from the noncommutative Khintchine inequality fails to capture correctly 
the structure of the problem.

We already sketched a proof of Lemma \ref{lem:wigner} using 
$\varepsilon$-nets in section \ref{sec:randpr}. We take the opportunity 
now to present another proof, due to Chevet \cite{Che78} and Gordon 
\cite{Gor85}, that provides a first illustration of the \emph{comparison 
methods} that will play an important role in the rest of this chapter. To 
this end, we first prove a classical comparison theorem for Gaussian 
processes due to Slepian and Fernique (see, e.g., \cite{BLM13}).

\begin{lemma}[Slepian-Fernique inequality]
\label{lem:slepian}
Let $Y\sim N(0,\Sigma^Y)$ and $Z\sim N(0,\Sigma^Z)$ be centered 
Gaussian vectors in $\mathbb{R}^n$. Suppose that
$$
	\mathbf{E}(Y_i-Y_j)^2 \le 
	\mathbf{E}(Z_i-Z_j)^2\quad\mbox{for all }1\le i,j\le n.
$$
Then
$$
	\mathbf{E}\max_{i\le n}Y_i \le
	\mathbf{E}\max_{i\le n}Z_i.
$$
\end{lemma}

\begin{proof}
Let $g,g'$ be independent standard Gaussian vectors. We can assume that 
$Y=(\Sigma^Y)^{1/2}g$ and $Z=(\Sigma^Z)^{1/2}g'$.  Let $Y(t) = 
\sqrt{t}Z+\sqrt{1-t}Y$ for $t\in[0,1]$. Then
\begin{align*}
	\frac{d}{dt}\mathbf{E}[f(Y(t))]
	&=
	\frac{1}{2}
	\mathbf{E}\bigg[
	\bigg\langle \nabla f(Y(t)),\frac{Z}{\sqrt{t}}-
	\frac{Y}{\sqrt{1-t}}\bigg\rangle\bigg] \\
	&=
	\frac{1}{2}
	\mathbf{E}\bigg[\frac{1}{\sqrt{t}}
	\bigg\langle (\Sigma^Z)^{1/2}\nabla f(Y(t)),g'\bigg\rangle
	-
	\frac{1}{\sqrt{1-t}}
	\bigg\langle (\Sigma^Y)^{1/2}\nabla f(Y(t)),g\bigg\rangle
	\bigg] \\
	&=
	\frac{1}{2}\sum_{i,j=1}^n
	(\Sigma^Z_{ij}-\Sigma^Y_{ij})\,
	\mathbf{E}\bigg[
	\frac{\partial^2f}{\partial x_i\partial x_j}(Y(t))
	\bigg],
\end{align*}
where we used Gaussian integration by parts in the last step.
We would really like to apply this identity with $f(x) = 
\max_ix_i$: if we can show that $\frac{d}{dt}\mathbf{E}[\max_i Y_i(t)]\ge 
0$, that would imply $\mathbf{E}[\max_iZ_i] = \mathbf{E}[\max_i 
Y_i(1)] \ge \mathbf{E}[\max_i Y_i(0)] = \mathbf{E}[\max_i Y_i]$ as 
desired. The problem is that the function $x\mapsto \max_ix_i$ is not 
sufficiently smooth: it does not possess second derivatives. We
therefore work with a smooth approximation.

Previously, we used $\|x\|_p$ as a smooth approximation of $\|x\|_\infty$.
Unfortunately, it turns out that Slepian-Fernique does \emph{not} hold 
when $\max_i Y_i$ and $\max_i Z_i$ are replaced by $\|Y\|_\infty$ and
$\|Z\|_\infty$, so this cannot work. We must therefore choose instead a 
\emph{one-sided} approximation.  In analogy with Remark \ref{rem:mxconc}, 
we choose
$$
	f_\beta(x) = 
	\frac{1}{\beta}
	\log\Bigg(
	\sum_{i=1}^n e^{\beta x_i}
	\Bigg).
$$
Clearly $\max_ix_i \le f_\beta(x) \le \max_ix_i + \beta^{-1}\log n$, 
so $f_\beta(x)\to\max_i x_i$ as $\beta\to\infty$. Also
$$
	\frac{\partial f_\beta}{\partial x_i}(x) =
	\frac{e^{\beta x_i}}{\sum_j e^{\beta x_j}} =:
	p_i(x),\qquad
	\frac{\partial^2 f_\beta}{\partial x_i\partial x_j}(x) =
	\beta\{\delta_{ij}p_i(x)-p_i(x)p_j(x)\},
$$
where we note that $p_i(x)\ge 0$ and $\sum_i p_i(x)=1$. The reader should 
check that
$$
	\frac{d}{dt}\mathbf{E}[f_\beta(Y(t))] =
	\frac{\beta}{4}\sum_{i\ne j}
	\{\mathbf{E}(Z_i-Z_j)^2-\mathbf{E}(Y_i-Y_j)^2\}\,
	\mathbf{E}[p_i(Y(t))p_j(Y(t))],
$$
which follows by rearranging the terms in the above expressions.
The right-hand side is nonnegative by assumption, and thus the proof
is easily completed.
\qed
\end{proof}

We can now prove Lemma \ref{lem:wigner}.

\begin{proof}[Proof of Lemma \ref{lem:wigner}]
That $\mathbf{E}\|X\|\gtrsim\sqrt{n}$ follows from Lemma 
\ref{lem:lowernck}, so it remains to prove 
$\mathbf{E}\|X\|\lesssim\sqrt{n}$. To this end, define $X_v := \langle 
v,Xv\rangle$ and $Y_v = 2\langle v,g\rangle$, where $g$ is a standard 
Gaussian vector. Then we can estimate
$$
	\mathbf{E}(X_v-X_w)^2 \le
	2\sum_{i,j=1}^n (v_iv_j-w_iw_j)^2
	\le 4 \|v-w\|^2 = \mathbf{E}(Y_v-Y_w)^2
$$
when $\|v\|=\|w\|=1$, where we used $1-\langle v,w\rangle^2 \le
2(1-\langle v,w\rangle)$ when $|\langle v,w\rangle| \le 1$.
It follows form the Slepian-Fernique lemma that we have
$$
	\mathbf{E}\lambda_{\rm max}(X) =
	\mathbf{E}\sup_{\|v\|=1}\langle v,Xv\rangle \le
	2\,\mathbf{E}\sup_{\|v\|=1}\langle v,g\rangle =
	2\,\mathbf{E}\|g\| \le 2\sqrt{n}.
$$
But as $X$ and $-X$ have the same distribution, so do the random 
variables $\lambda_{\rm max}(X)$ and $-\lambda_{\rm min}(X)=
\lambda_{\rm max}(-X)$. We can therefore estimate
$$
	\mathbf{E}\|X\| =
	\mathbf{E}(\lambda_{\rm max}(X)\vee{-\lambda_{\rm min}(X)}) 
	\le
	\mathbf{E}\lambda_{\rm max}(X)
	+ 2\,\mathbf{E}|\lambda_{\rm max}(X)-
	\mathbf{E}\lambda_{\rm max}(X)|
	= 2\sqrt{n}+O(1),
$$
where we used that $\mathrm{Var}(\lambda_{\rm max}(X))=O(1)$ by
Lemma \ref{lem:gaussconc}.
\qed\end{proof}
\end{example}

We have seen above two extreme examples: diagonal matrices and Wigner 
matrices. In the diagonal case, the noncommutative Khintchine inequality 
is sharp, while the lower bound in Lemma \ref{lem:lowernck} is suboptimal.  
On the other hand, for Wigner matrices, the noncommutative Khintchine 
inequality is suboptimal, while the lower bound in Lemma 
\ref{lem:lowernck} is sharp.  We therefore see that while the structural 
parameter $\sigma=\|\mathbf{E}X^2\|^{1/2}$ that appears in the 
noncommutative Khintchine inequality always crudely controls the spectral 
norm up to a logarithmic factor in the dimension, it fails to capture 
correctly the structure of the problem and cannot in general yield sharp 
bounds. The aim of the rest of this chapter is to develop a deeper 
understanding of the norms of structured random matrices that goes beyond 
Lemma \ref{lem:lowernck}.

\subsection{A second-order Khintchine inequality}

Having established that the noncommutative Khintchine inequality falls 
short of capturing the full structure of our random matrix model, we 
naturally aim to understand where things went wrong. The culprit is easy 
to identify. The main idea behind the proof of the noncommutative 
Khintchine inequality is that the case where the matrices $A_k$ commute is 
the worst possible, as is made precise by Lemma \ref{lem:exchg}. However, 
when the matrices $A_k$ do not commute, the behavior of the spectral norm 
can be \emph{strictly better} than is predicted by the noncommutative 
Khintchine inequality. The crucial shortcoming of the noncommutative 
Khintchine inequality is that it provides no mechanism to capture the 
effect of noncommutativity. 

\begin{remark}
This intuition is clearly visible in the examples of the previous section: 
the diagonal example corrsponds to choosing coefficient matrices $A_k$ of 
the form $e_ie_i^*$ for $1\le i\le n$, while to obtain a Wigner matrix we 
add additional coefficient matrices $A_k$ of the form $e_ie_j^*+e_je_i^*$ 
for $1\le i<j\le n$ (here $e_1,\ldots,e_n$ denotes the standard basis in 
$\mathbb{R}^n$).  Clearly the matrices $A_k$ commute in the diagonal 
example, in which case noncommutative Khintchine is sharp, but they do not 
commute for the Wigner matrix, in which case noncommutative Khintchine is 
suboptimal. 
\end{remark}

The present insight suggests that a good bound on the spectral norm of 
random matrices of the form $X=\sum_{k=1}^sg_kA_k$ should somehow take 
into account the algebraic structure of the coefficient matrices $A_k$. 
Unfortunately, it is not at all clear how this is to be accomplished. In 
this section we develop an interesting result in this spirit due to Tropp 
\cite{Tro15b}.  While this result is still very far from being sharp, the 
proof contains some interesting ideas, and provides at present the only 
known approach to improve on the noncommutative Khintchine inequality in 
the most general setting.

The intuition behind the result of Tropp is that the commutation inequality
$$
	\mathbf{E}[\mathrm{Tr}[A_kX^\ell A_k X^{2p-2-\ell}]] \le
	\mathbf{E}[\mathrm{Tr}[A_k^2X^{2p-2}]]
$$
of Lemma \ref{lem:exchg}, which captures the idea that the commutative
case is the worst case, should incur significant loss when the matrices
$A_k$ do not commute. Therefore, rather than apply this inequality 
directly, we should try to go to second order by integrating again by 
parts. For example, for the term $\ell=1$, we could write
\begin{align*}
	\mathbf{E}[\mathrm{Tr}[A_kXA_k X^{2p-3}]] &=
	\sum_{l=1}^s \mathbf{E}[g_l\mathrm{Tr}[A_kA_l A_k X^{2p-3}]]
	\\ &=
	\sum_{l=1}^s
	\sum_{m=0}^{2p-4}
	\mathbf{E}[\mathrm{Tr}[A_kA_l A_k X^m A_l X^{2p-4-m}]]. 
\end{align*}
If we could again permute the order of $A_l$ and $X^m$ on the right-hand 
side, we would obtain control of these terms not by the structural parameter 
$$
	\sigma=\Bigg\|\sum_{k=1}^sA_k^2\Bigg\|^{1/2}
$$
that appears in the noncommutative 
Khintchine inequality, but rather by the second-order ``noncommutative'' 
structural parameter
$$
	\Bigg\|\sum_{k,l=1}^s A_kA_lA_kA_l\Bigg\|^{1/4}.
$$
Of course, when the matrices $A_k$ commute, the latter parameter is equal 
to $\sigma$ and we recover the noncommutative Khintchine inequality; but 
when the matrices $A_k$ do not commute, it can be the case that this 
parameter is much smaller than $\sigma$.  This back-of-the-envelope
computation suggests that we might indeed hope to capture noncommutativity 
to some extent through the present approach.

In essence, this is precisely how we will proceed. However, there is a 
technical issue: the convexity that was exploited in the proof of Lemma 
\ref{lem:exchg} is no longer present in the second-order terms. We 
therefore cannot naively exchange $A_l$ and $X^m$ as suggested above, and 
the parameter $\|\sum_{k,l=1}^s A_kA_lA_kA_l\|^{1/4}$ is in fact too small 
to yield any meaningful bound (as is illustrated by a counterexample in 
\cite{Tro15b}).  The key idea in \cite{Tro15b} is that a classical complex 
analysis argument \cite[Appendix IX.4]{RS75} can be exploited to force 
convexity, at the expense of a larger second-order term.

\begin{theorem}[Tropp]
\label{thm:tropp}
Let $X=\sum_{k=1}^s g_kA_k$ as in the previous section. Define
$$
	\sigma := \Bigg\|\sum_{k=1}^sA_k^2\Bigg\|^{1/2},\qquad
	\tilde\sigma :=
	\sup_{U_1,U_2,U_3}
	\Bigg\|\sum_{k,l=1}^s A_k U_1A_l U_2A_k U_3 A_l\Bigg\|^{1/4},
$$
where the supremum is taken over all triples $U_1,U_2,U_3$ of commuting
unitary matrices.\footnote{For reasons that will become evident in
	the proof, it is essential to consider (complex) unitary matrices 
	$U_1,U_2,U_3$, despite that all the matrices $A_k$ and $X$ are 
	assumed to be real.
}
Then we have a second-order noncommutative Khintchine inequality
$$
	\mathbf{E}\|X\| \lesssim \sigma \, \log^{1/4} n + 
	\tilde\sigma \, \log^{1/2} n.
$$
\end{theorem}

Due to the (necessary) presence of the unitaries, the second-order 
parameter $\tilde\sigma$ is not so easy to compute.  It is verified in 
\cite{Tro15b} that $\tilde\sigma\le\sigma$ (so that Theorem 
\ref{thm:tropp} is no worse than the noncommutative Khintchine 
inequality), and that $\tilde\sigma=\sigma$ when the matrices $A_k$ 
commmute. On the other hand, an explicit computation in \cite{Tro15b} 
shows that if $X$ is a Wigner matrix as in Example \ref{ex:wigner}, we 
have $\sigma\asymp\sqrt{n}$ and $\tilde\sigma\asymp n^{1/4}$.  Thus 
Theorem \ref{thm:tropp} yields in this case $\mathbf{E}\|X\|\lesssim 
\sqrt{n}(\log n)^{1/4}$, which is strictly better than the noncommutative 
Khintchine bound $\mathbf{E}\|X\|\lesssim\sqrt{n}(\log n)^{1/2}$ but falls 
short of the sharp bound $\mathbf{E}\|X\|\asymp\sqrt{n}$.  We therefore 
see that Theorem \ref{thm:tropp} does indeed improve, albeit ever so 
slightly, on the noncommutative Khintchine bound.  The real interest of 
Theorem \ref{thm:tropp} is however the very general setting in which it 
holds, and that it does capture explicitly the noncommutativity of the 
coefficient matrices $A_k$.  In section \ref{sec:indep}, we will see that 
much sharper bounds can be obtained if we specialize to random matrices 
with independent entries. While this is perhaps the most interesting 
setting in practice, it will require us to depart from the much more 
general setting provided by the Khintchine-type inequalities that we have 
seen so far.

The remainder of this section is devoted to the proof of Theorem 
\ref{thm:tropp}. The proof follows essentially along the lines already 
indicated: we follow the proof of the noncommutative Khintchine inequality 
and integrate by parts a second time. The new idea in the proof is to 
understand how to appropriately extend Lemma \ref{lem:exchg}.

\begin{proof}[Proof of Theorem \ref{thm:tropp}]
We begin as in the proof of Theorem \ref{thm:nck} by writing
$$
	\mathbf{E}[\mathrm{Tr}[X^{2p}]] =
	\sum_{\ell=0}^{2p-2}
	\sum_{k=1}^s
	\mathbf{E}[\mathrm{Tr}[A_kX^\ell A_k X^{2p-2-\ell}]].
$$
Let us investigate each of the terms inside the first sum.

\textbf{Case $\ell=0,2p-2$.} In this case there is little to do: we
can estimate
$$
	\sum_{k=1}^s
        \mathbf{E}[\mathrm{Tr}[A_k^2 X^{2p-2}]]
	\le
	\mathrm{Tr}\Bigg[
        \Bigg(\sum_{k=1}^s A_k^2\Bigg)^p\Bigg]^{1/p}
	\mathbf{E}[\mathrm{Tr}[X^{2p}]]^{1-1/p}	
$$
precisely as in the proof of Theorem \ref{thm:nck}.

\textbf{Case $\ell=1,2p-3$.}
This is the first point at which something interesting happens.
Integrating by parts a second time as was discussed before
Theorem \ref{thm:tropp}, we obtain
$$
	\sum_{k=1}^s
	\mathbf{E}[\mathrm{Tr}[A_kX A_k X^{2p-3}]] 
	= \sum_{m=0}^{2p-4}
	\sum_{k,l=1}^s
	\mathbf{E}[\mathrm{Tr}[A_kA_lA_k X^m A_l X^{2p-4-m}]].
$$
The challenge we now face is to prove the appropriate analogue
of Lemma \ref{lem:exchg}.

\begin{lemma}
\label{lem:exchg1}
There exist unitary matrices $U_1,U_2$ (dependent on $X$ and $m$) such 
that
$$
	\sum_{k,l=1}^s\mathrm{Tr}[A_kA_lA_k X^m A_l X^{2p-4-m}] \le
	\Bigg|
	\sum_{k,l=1}^s\mathrm{Tr}[A_kA_lA_k U_1 A_l U_2 X^{2p-4}]
	\Bigg|.
$$
\end{lemma}

\begin{remark}
Let us start the proof as in Lemma \ref{lem:exchg} and see where things
go wrong.  In terms of the eigendecomposition 
$X=\sum_{i=1}^n\lambda_iv_iv_i^*$, we can write
$$
	\sum_{k,l=1}^s\mathrm{Tr}[A_kA_lA_k X^m A_l X^{2p-4-m}] =
	\sum_{k,l=1}^s\sum_{i,j=1}^n \lambda_i^m\lambda_j^{2p-4-m}
	\langle v_j,A_kA_lA_k v_i\rangle \langle v_i,A_lv_j\rangle.
$$
Unfortunately, unlike in the analogous expression in the proof of
Lemma \ref{lem:exchg}, the coefficients 
$\langle v_j,A_kA_lA_k v_i\rangle \langle v_i,A_lv_j\rangle$ can have
arbitrary sign. Therefore, we cannot easily force convexity of the above
expression as a function of $m$ as we did in Lemma \ref{lem:exchg}: if
we replace the terms in the sum by their absolute values, we will no
longer be able to interpret the resulting expression as a linear
algebraic object (a trace).

However, the above expression is still an
an \emph{analytic} function in the complex plane $\mathbb{C}$.
The idea that we will exploit is that analytic functions have
some hidden convexity built in, as we recall here without proof (cf.\ 
\cite[p.\,33]{RS75}).
\end{remark}

\begin{lemma}[Hadamard three line lemma]
If $\varphi:\mathbb{C}\to\mathbb{C}$ is analytic, the function
$t\mapsto\sup_{s\in\mathbb{R}}\log|\varphi(t+is)|$ is convex on the real 
line (provided it is finite).
\end{lemma}

\begin{proof}[Proof of Lemma \ref{lem:exchg1}]
We can assume that $X$
is nonsingular; otherwise we may replace $X$ by $X+\varepsilon$ and
let $\varepsilon\downarrow 0$ at the end of the proof.
Write $X=V|X|$ according
to its polar decomposition, and note that as $X$ is self-adjoint, 
$V=\mathrm{sign}(X)$ commutes with $X$ and therefore
$X^m=V^m|X|^m$. Define
$$
	\varphi(z) := 
	\sum_{k,l=1}^s\mathrm{Tr}[A_kA_lA_k V^m|X|^{(2p-4)z} A_l 
	V^{2p-4-m}
	|X|^{(2p-4)(1-z)}].
$$
As $X$ is nonsingular, $\varphi$ is analytic and $\varphi(t+is)$ 
is a periodic function of $s$ for every $t$. By the three line 
lemma, $\sup_{s\in\mathbb{R}}|\varphi(t+is)|$ attains its maximum for
$t\in[0,1]$ at either $t=0$ or $t=1$.  Moreover, the supremum itself is
attained at some $s\in\mathbb{R}$ by periodicity. We have therefore shown
that there exists $s\in\mathbb{R}$ such that
$$
	\Bigg|
	\sum_{k,l=1}^s\mathrm{Tr}[A_kA_lA_k X^m A_l X^{2p-4-m}]
	\Bigg|
	=
	\bigg|\varphi\bigg(\frac{m}{2p-4}\bigg)\bigg| \le
	|\varphi(is)|\vee|\varphi(1+is)|.
$$
But, for example,
$$
	|\varphi(is)| =
	\Bigg|
	\sum_{k,l=1}^s\mathrm{Tr}[A_kA_lA_k V^m|X|^{is(2p-4)} A_l 
	V^{2p-4-m}|X|^{-is(2p-4)}
	X^{2p-4}]\Bigg|,
$$
so if this term is the larger we can set $U_1=V^m|X|^{is(2p-4)}$
and $U_2=V^{2p-4-m}|X|^{-is(2p-4)}$ to obtain the statement of the
lemma (clearly $U_1$ and $U_2$ are unitary). 
If the term $|\varphi(1+is)|$ is larger, the claim follows 
in precisely the identical manner.
\qed\end{proof}

Putting together the above bounds, we obtain
\begin{align*}
	&\sum_{k=1}^s
	\mathbf{E}[\mathrm{Tr}[A_kX A_k X^{2p-3}]] 
	\\ &\le 
	(2p-3)\,
	\mathbf{E}\Bigg[
	\sup_{U_1,U_2}
	\Bigg|
	\sum_{k,l=1}^s\mathrm{Tr}[A_kA_lA_k U_1 A_l U_2 X^{2p-4}]
	\Bigg|\Bigg] \\
	&\le 
	(2p-3)
	\sup_{U}
	\mathrm{Tr}\Bigg[\Bigg|
	\sum_{k,l=1}^sA_kA_lA_k U A_l
	\Bigg|^{p/2}\Bigg]^{2/p}
	\mathbf{E}[\mathrm{Tr}[X^{2p}]]^{1-2/p}.
\end{align*}
This term will evidently yield a term of order $\tilde\sigma$
when $p\sim\log n$.

\textbf{Case $2\le\ell\le 2p-4$.} These terms are dealt with much in the
same way as in the previous case, except the computation is a bit more 
tedious. As we have come this far, we might as well complete the argument.
We begin by noting that
$$
	\sum_{k=1}^s
	\mathbf{E}[\mathrm{Tr}[A_kX^\ell A_k X^{2p-2-\ell}]]
	\le
	\sum_{k=1}^s
	\mathbf{E}[\mathrm{Tr}[A_kX^2 A_k X^{2p-4}]]
$$
for every $2\le\ell\le 2p-4$. This follows by convexity precisely in the 
same way as in Lemma \ref{lem:exchg}, and we omit the (identical) proof.
To proceed, we integrate by parts:
\begin{align*}
	&\sum_{k=1}^s
	\mathbf{E}[\mathrm{Tr}[A_kX^2 A_k X^{2p-4}]] =
	\sum_{k,l=1}^s
	\mathbf{E}[g_l\mathrm{Tr}[A_kA_l X A_k X^{2p-4}]] \\
	&=
	\sum_{k,l=1}^s
	\mathbf{E}[\mathrm{Tr}[A_kA_l^2A_k X^{2p-4}]]+
	\sum_{m=0}^{2p-5}
	\sum_{k,l=1}^s
	\mathbf{E}[\mathrm{Tr}[A_kA_l X A_k X^m A_l X^{2p-5-m}]].
\end{align*}
We deal separately with the two types of terms.

\begin{lemma}
\label{lem:exchg2}
There exist unitary matrices $U_1,U_2,U_3$ such that
$$
	\sum_{k,l=1}^s
	\mathrm{Tr}[A_kA_l X A_k X^m A_l X^{2p-5-m}] \le
	\Bigg|
	\sum_{k,l=1}^s
	\mathrm{Tr}[A_kA_l U_1 A_k U_2 A_l U_3 X^{2p-4}]
	\Bigg|.
$$
\end{lemma}

\begin{proof}
Let $X=V|X|$ be the polar decomposition of $X$, and define
$$
	\varphi(y,z) :=
	\sum_{k,l=1}^s
	\mathbf{E}[\mathrm{Tr}[A_kA_l V|X|^{(2p-4)y} A_k V^m
	|X|^{(2p-4)z} A_l 
	V^{2p-5-m}|X|^{(2p-4)(1-y-z)}]].	
$$
Now apply the three line lemma to $\varphi$ twice: 
to $\varphi(\cdot,z)$ with $z$ fixed, then to
$\varphi(y,\cdot)$ with $y$ fixed.  The omitted details are almost 
identical
to the proof of Lemma \ref{lem:exchg1}.
\qed\end{proof}

\begin{lemma}
\label{lem:exchg3}
We have for $p\ge 2$
$$
	\sum_{k,l=1}^s\mathrm{Tr}[A_kA_l^2A_k X^{2p-4}] \le
	\mathrm{Tr}\Bigg[
        \Bigg(\sum_{k=1}^s A_k^2\Bigg)^p\Bigg]^{2/p}
	\mathrm{Tr}[X^{2p}]^{1-2/p}.
$$
\end{lemma}

\begin{proof}
We argue essentially as in Lemma \ref{lem:exchg}. Define $H =
\sum_{l=1}^sA_l^2$ and let
$$
        \varphi(z) :=
        \sum_{k=1}^s
        \mathrm{Tr}[A_k H^{(p-1)z} A_k |X|^{(2p-2)(1-z)}],
$$
so that the quantity we would like to bound is $\varphi(1/(p-1))$.
By expressing $\varphi(z)$ in terms of the spectral decompositions 
$X=\sum_{i=1}^n\lambda_iv_iv_i^*$ and $H=\sum_{i=1}^n\mu_iw_iw_i^*$,
we can verify by explicit computation that $z\mapsto\log\varphi(z)$ is 
convex on $z\in[0,1]$.  Therefore
$$
	\varphi(1/(p-1)) \le
	\varphi(1)^{1/(p-1)}
	\varphi(0)^{(p-2)/(p-1)} =
	\mathrm{Tr}[H^p]^{1/(p-1)}
	\mathrm{Tr}[HX^{2p-2}]^{(p-2)/(p-1)}.
$$
But $\mathrm{Tr}[H |X|^{2p-2}]\le \mathrm{Tr}[H^p]^{1/p} 
\mathrm{Tr}[X^{2p}]^{1-1/p}$ by H\"older's inequality, and the
conclusion follows readily by substituting this into the above
expression.
\qed\end{proof}

Putting together the above bounds and using H\"older's inequality yields
\begin{multline*}
	\sum_{k=1}^s
	\mathbf{E}[\mathrm{Tr}[A_kX^\ell A_k X^{2p-2-\ell}]]
	\le
	\mathrm{Tr}\Bigg[
        \Bigg(\sum_{k=1}^s A_k^2\Bigg)^p\Bigg]^{2/p}
	\mathbf{E}[\mathrm{Tr}[X^{2p}]]^{1-2/p} 
	\\ \mbox{}
	+
	(2p-4)\sup_{U_1,U_2}
	\mathrm{Tr}\Bigg[
	\Bigg|\sum_{k,l=1}^s
	A_kA_l U_1 A_k U_2 A_l\Bigg|^{p/2}\Bigg]^{2/p}
	\mathbf{E}[\mathrm{Tr}[X^{2p}]]^{1-2/p}.
\end{multline*}

\textbf{Conclusion.} Let $p\asymp\log n$.  Collecting the above bounds 
yields
$$
	\mathbf{E}[\mathrm{Tr}[X^{2p}]] \lesssim
	\sigma^2\mathbf{E}[\mathrm{Tr}[X^{2p}]]^{1-1/p} +
	p\,(\sigma^4 + p\,\tilde\sigma^4)\,
	\mathbf{E}[\mathrm{Tr}[X^{2p}]]^{1-2/p},
$$
where we used Lemma \ref{lem:mommeth} to simplify the constants.
Rearranging gives
$$
	\mathbf{E}[\mathrm{Tr}[X^{2p}]]^{2/p} \lesssim
	\sigma^2\mathbf{E}[\mathrm{Tr}[X^{2p}]]^{1/p} +
	p\,(\sigma^4 + p\,\tilde\sigma^4),
$$
which is a simple quadratic inequality for 
$\mathbf{E}[\mathrm{Tr}[X^{2p}]]^{1/p}$.  Solve this inequality using 
the quadratic formula and apply again Lemma \ref{lem:mommeth} to
conclude the proof. \qed\end{proof}

\section{Matrices with independent entries}
\label{sec:indep}

The Khintchine-type inequalities developed in the previous section have 
the advantage that they can be applied in a remarkably general setting: 
they not only allow an arbitrary variance pattern of the entries, but even 
an arbitrary dependence structure between the entries. This makes such 
bounds useful in a wide variety of situations. Unfortunately, we have also 
seen that Khintchine-type inequalities yield suboptimal bounds already in 
the simplest examples: the mechanism behind the proofs of these 
inequalities is too crude to fully capture the structure of the underlying 
random matrices at this level of generality. In order to gain a deeper 
understanding, we must impose some additional structure on the matrices 
under consideration.

In this section, we specialize to what is perhaps the most important case 
of the random matrices investigated in the previous section: we consider
symmetric random matrices with \emph{independent} entries.  More 
precisely, in most of this section, we will study the following basic 
model.  Let $g_{ij}$ be independent standard Gaussian random variables
and let $b_{ij}\ge 0$ be given scalars for $i\ge j$.  We consider the 
$n\times n$ symmetric random matrix $X$ whose entries are given by 
$X_{ij}=b_{ij}g_{ij}$, that is,
$$
	X = \left[
        \begin{matrix}
        b_{11}g_{11} & b_{12}g_{12} & \cdots & b_{1n}g_{1n} \\
        b_{12}g_{12} & b_{22}g_{22} &        & b_{2n}g_{2n} \\
        \vdots &        & \ddots & \vdots \\
        b_{1n}g_{1n} & b_{2n}g_{2n} & \cdots & b_{nn}g_{nn} \\
        \end{matrix}
        \right].
$$
In other words, $X$ is the symmetric random matrix whose entries
above the diagonal are independent Gaussian variables $X_{ij}\sim 
N(0,b_{ij}^2)$, where the structure of the matrix is controlled by the 
given variance pattern $\{b_{ij}\}$.  As the matrix is symmetric, we will 
write for simplicity $g_{ji}=g_{ij}$ and $b_{ji}=b_{ij}$ in the sequel.

The present model differs from the model of the previous section only to 
the extent that we imposed the additional independence assumption on the 
entries. In particular, the 
noncommutative Khintchine inequality reduces in this setting to
$$
	\mathbf{E}\|X\| \lesssim
	\max_{i\le n}
	\sqrt{\sum_{j=1}^n b_{ij}^2}
	\,\sqrt{\log n},
$$
while Theorem \ref{thm:tropp} yields (after some tedious computation)
$$
	\mathbf{E}\|X\| \lesssim
	\max_{i\le n}
	\sqrt{\sum_{j=1}^n b_{ij}^2} ~
	(\log n)^{1/4}
	+
	\max_{i\le n}
	\Bigg(\sum_{j=1}^n b_{ij}^4\Bigg)^{1/4}
	\sqrt{\log n}.
$$
Unfortunately, we have already seen that neither of these results is 
sharp even for Wigner matrices (where $b_{ij}=1$ for all $i,j$).  The aim 
of this section is to develop much sharper inequalities for matrices with 
independent entries that capture \emph{optimally} in many cases the 
underlying structure. The independence assumption will be crucially 
exploited to control the structure of these matrices, and it is an 
interesting open problem to understand to what extent the mechanisms 
developed in this section persist in the presence of dependence between 
the entries (cf.\ section \ref{sec:open}).

\subsection{Lata{\l}a's inequality and beyond}

The earliest nontrivial result on the spectral norm Gaussian random 
matrices with independent entries is the following inequality due to 
Lata{\l}a \cite{Lat05}.

\begin{theorem}[Lata{\l}a]
\label{thm:latala}
In the setting of this section, we have
$$
	\mathbf{E}\|X\| \lesssim
	\max_{i\le n}\sqrt{\sum_{j=1}^n b_{ij}^2} +
	\Bigg(\sum_{i,j=1}^n b_{ij}^4\Bigg)^{1/4}.
$$
\end{theorem}

Lata{\l}a's inequality yields a sharp bound 
$\mathbf{E}\|X\|\lesssim\sqrt{n}$ for Wigner matrices, but is already 
suboptimal for the diagonal matrix of Example \ref{ex:diag} where the 
resulting bound $\mathbf{E}\|X\|\lesssim n^{1/4}$ is very far from the 
correct answer $\mathbf{E}\|X\|\asymp\sqrt{\log n}$. In this sense, we see 
that Theorem \ref{thm:latala} fails to correctly capture the structure of 
the underlying matrix. Lata{\l}a's inequality is therefore not too useful 
for \emph{structured} random matrices; it has however been widely applied 
together with a simple symmetrization argument \cite[Theorem 
2]{Lat05} to show that the sharp bound $\mathbf{E}\|X\|\asymp\sqrt{n}$
remains valid for Wigner matrices with general (non-Gaussian) 
distribution of the entries.

In this section, we develop a nearly sharp improvement of Lata{\l}a's 
inequality that can yield optimal results for many structured random 
matrices.

\begin{theorem}[\cite{vH15}]
\label{thm:gine}
In the setting of this section, we have
$$
	\mathbf{E}\|X\| \lesssim
	\max_{i\le n}\sqrt{\sum_{j=1}^n b_{ij}^2} +
	\max_{i\le n}
	\Bigg(\sum_{j=1}^n b_{ij}^4\Bigg)^{1/4}\sqrt{\log i}.
$$
\end{theorem}

Let us first verify that Lata{\l}a's inequality does indeed follow.

\begin{proof}[Proof of Theorem \ref{thm:latala}]
As the matrix norm $\|X\|$ is unchanged if we permute the rows and columns 
of $X$, we may assume without loss of generality that $\sum_{j=1}^n 
b_{ij}^4$ is decreasing in $i$ (this choice minimizes the upper bound in 
Theorem \ref{thm:gine}).  Now recall the following elementary fact:
if $x_1\ge x_2\ge\cdots\ge x_n\ge 0$, then $x_k\le 
\frac{1}{k}\sum_{i=1}^nx_i$ for every $k$.  In the present case, this
observation and Theorem \ref{thm:gine} imply
$$
        \mathbf{E}\|X\| \lesssim
        \max_{i\le n}\sqrt{\sum_{j=1}^n b_{ij}^2} +
        \Bigg(\sum_{i,j=1}^n b_{ij}^4\Bigg)^{1/4}
        \max_{1\le i<\infty}\frac{\sqrt{\log i}}{i^4},
$$
which concludes the proof of Theorem \ref{thm:latala}.
\qed\end{proof}

The inequality of Theorem \ref{thm:gine} is somewhat reminiscent of the 
bound obtained in the present setting from Theorem \ref{thm:tropp}, with a 
crucial difference: there is no logarithmic factor in front of the first 
term. As we already proved in Lemma \ref{lem:lowernck} that
$$
	\mathbf{E}\|X\| \gtrsim
        \max_{i\le n}\sqrt{\sum_{j=1}^n b_{ij}^2},
$$
we see that Theorem \ref{thm:gine} provides an \emph{optimal} bound 
whenever the first term dominates, which is the case for a wide range of 
structured random matrices.  To get a feeling for the sharpness of Theorem 
\ref{thm:gine}, let us consider an illuminating example.

\begin{example}[Block matrices]
\label{ex:block}
Let $1\le k\le n$ and suppose for simplicity that $n$ is divisible by $k$.
We consider the $n\times n$ symmetric block-diagonal matrix $X$ of the 
form
$$
	X = \left[
        \begin{matrix}
	\mathbf{X}_1 &              &        &  \\
	             & \mathbf{X}_2 &        &  \\
                     &              & \ddots &  \\
                     &              &        & \mathbf{X}_{n/k} \\
	\end{matrix}
	\right],
$$
where $\mathbf{X}_1,\ldots,\mathbf{X}_{n/k}$ are independent $k\times k$ 
Wigner matrices.  This model interpolates between the diagonal matrix of 
Example \ref{ex:diag} (the case $k=1$) and the Wigner matrix of Example 
\ref{ex:wigner} (the case $k=n$).  Note that 
$\|X\|=\max_i\|\mathbf{X}_i\|$, so we can compute
$$
	\mathbf{E}\|X\| \lesssim
	\mathbf{E}[\|\mathbf{X}_1\|^{\log n}]^{1/\log n} \le
	\mathbf{E}\|\mathbf{X}_1\| +
	\mathbf{E}[(\|\mathbf{X}_1\|-\mathbf{E}\|\mathbf{X}_1\|)^{\log n}]^{1/\log n}
	\lesssim
	\sqrt{k}+\sqrt{\log n}
$$
using Lemmas \ref{lem:lplinf}, \ref{lem:wigner}, and \ref{lem:gaussconc}, 
respectively.  On the other hand, Lemma \ref{lem:lowernck} implies that
$\mathbf{E}\|X\|\gtrsim\sqrt{k}$, while we can trivially estimate
$\mathbf{E}\|X\|\ge \mathbf{E}\max_iX_{ii} \asymp\sqrt{\log n}$.
Averaging these two lower bounds, we have evidently shown that
$$
	\mathbf{E}\|X\| \asymp \sqrt{k} + \sqrt{\log n}.
$$
This explicit computation provides a simple but very useful benchmark 
example for testing inequalities for structured random matrices.

In the present case, applying Theorem \ref{thm:gine} to this example 
yields
$$
	\mathbf{E}\|X\| \lesssim \sqrt{k} + k^{1/4}\sqrt{\log n}.
$$
Therefore, in the present example, Theorem \ref{thm:gine} fails to be 
sharp only when $k$ is in the range $1\ll k\ll (\log n)^2$.  This 
suboptimal parameter range will be completely eliminated by the sharp 
bound to be proved in section \ref{sec:bvh} below. But the bound of Theorem
\ref{thm:gine} is already sharp in the vast majority of cases, and is
of significant interest in its own right for reasons that will be 
discussed in detail in section \ref{sec:open}.
\end{example}

An important feature of the inequalities of this section should be 
emphasized: unlike all bounds we have encountered so far, the present 
bounds are \emph{dimension-free}. As was discussed in Remark 
\ref{rem:badmoment}, one cannot expect to obtain sharp dimension-free 
bounds using the moment method, and it therefore comes as no surprise 
that the bounds of the present section will therefore be obtained by the 
random process method. The original proof of Lata{\l}a proceeds by a 
difficult and very delicate explicit construction of a multiscale net in 
the spirit of Remark \ref{rem:gench}. We will follow here a much simpler 
approach that was developed in \cite{vH15} to prove Theorem 
\ref{thm:gine}.

The basic idea behind our approach was already encountered in the proof 
of Lemma \ref{lem:wigner} to bound the norm of a Wigner matrix (where 
$b_{ij}=1$ for all $i,j$): we seek a Gaussian process $Y_v$ that dominates
the process $X_v:=\langle v,Xv\rangle$ whose supremum coincides with the
spectral norm.  The present setting is significantly more challenging, 
however. To see the difficulty, let us try to adapt directly the
proof of Lemma \ref{lem:wigner} to the present structured setting: we
readily compute
$$
	\mathbf{E}(X_v-X_w)^2
	\le
	2\sum_{i,j=1}^n b_{ij}^2(v_iv_j-w_iw_j)^2
	\le
	4\max_{i,j\le n}b_{ij}^2\,\|v-w\|^2.
$$
We can therefore dominate $X_v$ by the Gaussian process
$Y_v=2\max_{i,j}b_{ij}\,\langle v,g\rangle$.  Proceeding as in the proof
of Lemma \ref{lem:wigner}, this yields the following upper bound:
$$
	\mathbf{E}\|X\| \lesssim \max_{i,j\le n}b_{ij}\sqrt{n}.
$$
This bound is sharp for Wigner matrices (in this case the present proof
reduces to that of Lemma \ref{lem:wigner}), but is woefully inadequate in
any structured example.
The problem with the above bound is that it always crudely estimates the 
behavior of the increments $\mathbf{E}[(X_v-X_w)]^{1/2}$ by a 
\emph{Euclidean} norm $\|v-w\|$, regardless of the structure of the 
underlying matrix. However, the geometry defined by 
$\mathbf{E}[(X_v-X_w)]^{1/2}$ depends strongly on the structure of the 
matrix, and is typically highly non-Euclidean. For example, in the 
diagonal matrix of Example \ref{ex:diag}, we have 
$\mathbf{E}[(X_v-X_w)]^{1/2} = \|v^2-w^2\|$ where $(v^2)_i:=v_i^2$. As 
$v^2$ is in the simplex whenever $v\in B$, we see that the underlying 
geometry in this case is that of an $\ell_1$-norm and not of an 
$\ell_2$-norm. In more general examples, however, it is far from clear
what is the correct geometry.

The key challenge we face is to design a comparison process that is easy 
to bound, but whose geometry nonetheless captures faithfully the 
structure of the underlying matrix. To develop some intuition for how 
this might be accomplished, let us consider in first instance instead of 
the increments $\mathbf{E}[(X_v-X_w)^2]^{1/2}$ only the standard deviation
$\mathbf{E}[X_v^2]^{1/2}$ of the process $X_v=\langle 
v,Xv\rangle$. We easily compute
$$
	\mathbf{E} X_v^2 =
	2\sum_{i\ne j}v_i^2b_{ij}^2v_j^2 +
	\sum_{i=1}^n b_{ii}^2v_i^4 \le
	2\sum_{i=1}^n x_i(v)^2,
$$
where we defined the nonlinear map $x:\mathbb{R}^n\to\mathbb{R}^n$ as
$$
	x_i(v) := v_i\sqrt{\sum_{j=1}^n b_{ij}^2v_j^2}.
$$
This computation suggests that we might attempt to dominate the process 
$X_v$ by the process $Y_v = \langle x(v),g\rangle$, whose increments
$\mathbf{E}[(Y_v-Y_w)^2]^{1/2}=\|x(v)-x(w)\|$ capture the non-Euclidean
nature of the underlying geometry through the nonlinear map
$x$. The reader may readily verify, for example, that the latter
process captures automatically the correct geometry of our two extreme
examples of Wigner and diagonal matrices.

Unfortunately, the above choice of comparison process $Y_v$ is too
optimistic: while we have chosen this process so that
$\mathbf{E}X_v^2 \lesssim \mathbf{E}Y_v^2$ by construction, the 
Slepian-Fernique inequality requires the stronger bound
$\mathbf{E}(X_v-X_w)^2\lesssim \mathbf{E}(Y_v-Y_w)^2$.  It turns out
that the latter inequality does not always hold \cite{vH15}. However,
the inequality \emph{nearly} holds, which is the key observation behind
the proof of Theorem \ref{thm:gine}.

\begin{lemma}
\label{lem:ginecomp}
For every $v,w\in\mathbb{R}^n$
$$
	\mathbf{E}(\langle v,Xv\rangle-\langle w,Xw\rangle)^2 \le
	4\|x(v)-x(w)\|^2 - \sum_{i,j=1}^n (v_i^2-w_i^2)b_{ij}^2
	(v_j^2-w_j^2).
$$
\end{lemma}

\begin{proof}
We simply compute both sides and compare.  Define for simplicity
the seminorm $\|\cdot\|_i$ as
$\|v\|_i^2 := \sum_{j=1}^n b_{ij}^2v_j^2$,
so that $x_i(v)=v_i\|v\|_i$.
First, we note that
\begin{align*}
	\mathbf{E}(\langle v,Xv\rangle-\langle w,Xw\rangle)^2 &=
	\mathbf{E}\langle v+w,X(v-w)\rangle^2 \\
	&=
	\sum_{i=1}^n(v_i-w_i)^2\|v+w\|_i^2 +
	\sum_{i,j=1}^n (v_i^2-w_i^2)b_{ij}^2(v_j^2-w_j^2).
\end{align*}
On the other hand, as $2(x_i(v)-x_i(w))=(v_i+w_i)(\|v\|_i-\|w\|_i)+
(v_i-w_i)(\|v\|_i+\|w\|_i)$,
\begin{align*}
	4\|x(v)-x(w)\|^2 &= 
	\sum_{i=1}^n (v_i+w_i)^2(\|v\|_i-\|w\|_i)^2 +
	\sum_{i=1}^n (v_i-w_i)^2(\|v\|_i+\|w\|_i)^2 \\
	&\qquad\mbox{}
	+2\sum_{i,j=1}^n (v_i^2-w_i^2)b_{ij}^2(v_j^2-w_j^2).	
\end{align*}
The result follows readily from the triangle inequality
$\|v+w\|_i\le\|v\|_i+\|w\|_i$.
\qed\end{proof}

We can now complete the proof of Theorem \ref{thm:gine}.

\begin{proof}[Proof of Theorem \ref{thm:gine}]
Define the Gaussian processes
$$
	X_v=\langle v,Xv\rangle,\qquad\quad
	Y_v = 2\langle x(v),g\rangle + \langle v^2,Y\rangle, 
$$
where $g\sim N(0,I)$ is a standard Gaussian vector in $\mathbb{R}^n$, 
$(v^2)_i:=v_i^2$, and $Y\sim N(0,B^-)$ is a centered Gaussian vector
that is independent of $g$ and whose
covariance matrix $B^-$ is the negative part of the matrix of 
variances $B=(b_{ij}^2)$ (if $B$ has eigendecomposition
$B=\sum_i\lambda_iv_iv_i^*$, the negative part $B^-$ is defined as
$B^-=\sum_i\max(-\lambda_i,0)v_iv_i^*$).
As $-B\preceq B^-$ by construction, it is readily seen that Lemma
\ref{lem:ginecomp} implies
$$
	\mathbf{E}(X_v-X_w)^2 \le
	4\|x(v)-x(w)\|^2 + \langle v^2-w^2,B^-(v^2-w^2)\rangle =
	\mathbf{E}(Y_v-Y_w)^2.
$$
We can therefore argue by the Slepian-Fernique inequality that
$$
	\mathbf{E}\|X\| \lesssim
	\mathbf{E}\sup_{v\in B}Y_v \le
	2\,\mathbf{E}\sup_{v\in B}\langle x(v),g\rangle +
	\mathbf{E}\max_{i\le n}Y_i
$$
as in the proof of Lemma \ref{lem:wigner}.  It remains to bound each term
on the right.

Let us begin with the second term. Using the moment method
as in section \ref{sec:mommeth}, one obtains the dimension-dependent bound
$\mathbf{E}\max_{i}Y_i\lesssim \max_i\mathrm{Var}(Y_i)^{1/2}\sqrt{\log n}$.
This bound is sharp when all the variances $\mathrm{Var}(Y_{i})=B_{ii}^-$ 
are of the same order, but can be suboptimal when many of the variances
are small.  Instead, we will use a sharp \emph{dimension-free} bound on 
the maximum of Gaussian random variables.

\begin{lemma}[Subgaussian maxima]
\label{lem:gaussmax}
Suppose that $g_1,\ldots,g_n$ satisfy 
$\mathbf{E}[|g_i|^k]^{1/k}\lesssim\sqrt{k}$ for all $k,i$, 
and let $\sigma_1,\ldots,\sigma_n\ge 0$.  Then we have
$$
	\mathbf{E}\max_{i\le n}|\sigma_ig_i| \lesssim
	\max_{i\le n}\sigma_i\sqrt{\log(i+1)}.
$$
\end{lemma}

\begin{proof}
By a union bound and Markov's inequality
$$
	\mathbf{P}\bigg[\max_{i\le n}|\sigma_ig_i| \ge t\bigg] \le
	\sum_{i=1}^n\mathbf{P}[|\sigma_ig_i| \ge t] \le
	\sum_{i=1}^n \bigg(
	\frac{\sigma_i\sqrt{2\log(i+1)}}{t}
	\bigg)^{2\log(i+1)}.
$$
But we can estimate
$$
	\sum_{i=1}^\infty s^{-2\log(i+1)} = 
	\sum_{i=1}^\infty (i+1)^{-2}(i+1)^{-2\log s+2}
	\le
	2^{-2\log s+2}
	\sum_{i=1}^\infty (i+1)^{-2}
	\lesssim 
	s^{-2\log 2}
$$
as long as $\log s>1$. Setting $s=
t/\max_i\sigma_i\sqrt{2\log(i+1)}$, we obtain
\begin{align*}
	\mathbf{E}\max_{i\le n}|\sigma_ig_i| &=
	\max_i\sigma_i\sqrt{2\log(i+1)}
	\int_0^\infty 
	\mathbf{P}\bigg[\max_{i\le n}|\sigma_ig_i| \ge s
	\max_i\sigma_i\sqrt{2\log(i+1)}\bigg]\,ds
	\\
	&\lesssim
	\max_i\sigma_i\sqrt{2\log(i+1)}\,
	\bigg(e +
	\int_{e}^\infty s^{-2\log 2}ds\bigg)
	\lesssim
	\max_i\sigma_i\sqrt{\log(i+1)},
\end{align*}
which completes the proof.
\qed\end{proof}

\begin{remark}
\label{rem:gaussmax}
Lemma \ref{lem:gaussmax} does not require the variables $g_i$ to be
either independent or Gaussian.  However, if $g_1,\ldots,g_n$ are
independent standard Gaussian variables (which satisfy 
$\mathbf{E}[|g_i|^k]^{1/k}\lesssim\sqrt{k}$ by Lemma \ref{lem:gaussmom})
and if $\sigma_1\ge\sigma_2\ge\cdots\ge\sigma_n>0$ (which is the ordering
that optimizes the bound of Lemma \ref{lem:gaussmax}),
then 
$$
	\mathbf{E}\max_{i\le n}|\sigma_ig_i| \asymp
	\max_{i\le n}\sigma_i\sqrt{\log(i+1)},
$$
cf.\ \cite{vH15}.
This shows that Lemma \ref{lem:gaussmax} captures precisely the dimension-free
behavior of the maximum of independent centered Gaussian variables.
\end{remark}

To estimate the second term in our bound on $\mathbf{E}\|X\|$,
note that $(B^-)^2\preceq B^2$ implies
$$
	\mathrm{Var}(Y_i)^2 =
	(B_{ii}^-)^2 \le 
	(B^-)^2_{ii} \le
	(B^2)_{ii} =
	\sum_{j=1}^n (B_{ij})^2 =
	\sum_{j=1}^n b_{ij}^4.
$$
Applying Lemma \ref{lem:gaussmax} with $g_i=Y_i/\mathrm{Var}(Y_i)^{1/2}$
yields the bound
$$
	\mathbf{E}\max_{i\le n}Y_i \lesssim
	\max_{i\le n}
        \Bigg(\sum_{j=1}^n b_{ij}^4\Bigg)^{1/4}\sqrt{\log(i+1)}.
$$
Now let us estimate the first term in our bound on $\mathbf{E}\|X\|$.
Note that
$$
	\sup_{v\in B}\langle x(v),g\rangle =
	\sup_{v\in B}\sum_{j=1}^n
	g_jv_j\sqrt{\sum_{i=1}^n v_i^2b_{ij}^2} \le
	\sup_{v\in B}\sqrt{\sum_{i,j=1}^n
	v_i^2b_{ij}^2g_j^2} =
	\max_{i\le n}\sqrt{\sum_{j=1}^n b_{ij}^2g_j^2},
$$
where we used Cauchy-Schwarz and the fact that $v^2$ is in the $\ell_1$-ball
whenever $v$ is in the $\ell_2$-ball.  We can therefore estimate,
using Lemma \ref{lem:gaussconc} and Lemma \ref{lem:gaussmax},
\begin{align*}
	\mathbf{E}\sup_{v\in B}\langle x(v),g\rangle &\le
	\max_{i\le n}\mathbf{E}\sqrt{\sum_{j=1}^n b_{ij}^2g_j^2}
	+
	\mathbf{E}\max_{i\le n}\Bigg|
	\sqrt{\sum_{j=1}^n b_{ij}^2g_j^2}-
	\mathbf{E}\sqrt{\sum_{j=1}^n b_{ij}^2g_j^2}
	\Bigg|
	\\
	&\lesssim
	\max_{i\le n}\sqrt{\sum_{j=1}^nb_{ij}^2} +
	\max_{i,j\le n}b_{ij}\sqrt{\log(i+1)}.
\end{align*}
Putting everything together gives
$$
	\mathbf{E}\|X\| \lesssim
	\max_{i\le n}\sqrt{\sum_{j=1}^nb_{ij}^2} +
	\max_{i,j\le n}b_{ij}\sqrt{\log(i+1)} +
	\max_{i\le n}
        \Bigg(\sum_{j=1}^n b_{ij}^4\Bigg)^{1/4}\sqrt{\log(i+1)}.
$$
It is not difficult to simplify this (at the expense of a larger
universal constant) to obtain the bound in the statement of Theorem
\ref{thm:gine}.
\qed\end{proof}


\subsection{A sharp dimension-dependent bound}
\label{sec:bvh}

The approach developed in the previous section yields optimal results 
for many structured random matrices with independent entries. The 
crucial improvement of Theorem \ref{thm:gine} over the noncommutative 
Khintchine inequality is that no logarithmic factor appears in the first
term.  Therefore, when this term dominates, Theorem \ref{thm:gine} is 
sharp by Lemma \ref{lem:lowernck}. However, the second term in Theorem 
\ref{thm:gine} is not quite sharp, as is illustrated in Example 
\ref{ex:block}.  While Theorem \ref{thm:gine} captures much of the 
geometry of the underlying model, there remains some residual 
inefficiency in the proof.

In this section, we will develop an improved version of Theorem 
\ref{thm:gine} that is essentially sharp (in a sense that will be made 
precise below). Unfortunately, it is not known at present how such a 
bound can be obtained using the random process method, and we revert 
back to the moment method in the proof. The price we pay for this is 
that we lose the dimension-free nature of Theorem \ref{thm:gine}.

\begin{theorem}[\cite{BvH16}]
\label{thm:bvh}
In the setting of this section, we have
$$
	\mathbf{E}\|X\| \lesssim
	\max_{i\le n}\sqrt{\sum_{j=1}^n b_{ij}^2} +
	\max_{i,j\le n}b_{ij}\sqrt{\log n}.
$$
\end{theorem}

To understand why this result is sharp, let us recall (Remark 
\ref{rem:badmoment}) that the moment method necessarily bounds not the 
quantity $\mathbf{E}\|X\|$, but rather the larger quantity 
$\mathbf{E}[\|X\|^{\log n}]^{1/\log n}$.  The latter quantity is now 
however completely understood.

\begin{corollary}
\label{cor:bvhlower}
In the setting of this section, we have
$$
	\mathbf{E}[\|X\|^{\log n}]^{1/\log n} \asymp
	\max_{i\le n}\sqrt{\sum_{j=1}^n b_{ij}^2} +
	\max_{i,j\le n}b_{ij}\sqrt{\log n}.
$$
\end{corollary}

\begin{proof}
The upper bound follows from the proof of Theorem \ref{thm:bvh}.
The first term on the right
is a lower bound by Lemma \ref{lem:lowernck}. On the other hand,
if $b_{kl}=\max_{i,j}b_{ij}$, then
$\mathbf{E}[\|X\|^{\log n}]^{1/\log n}\ge
\mathbf{E}[|X_{kl}|^{\log n}]^{1/\log n}\gtrsim b_{kl}\sqrt{\log n}$
as $\mathbf{E}[|X_{kl}|^p]^{1/p}\asymp b_{kl}\sqrt{p}$.
\qed\end{proof}

The above result shows that Theorem \ref{thm:bvh} is in fact the 
\emph{optimal} result that could be obtained by the moment method. This 
result moreover yields optimal bounds even for $\mathbf{E}\|X\|$ in 
almost all situations of practical interest, as it is true under mild 
assumptions that $\mathbf{E}\|X\|\asymp \mathbf{E}[\|X\|^{\log 
n}]^{1/\log n}$ (as will be discussed in section \ref{sec:open}). 
Nonetheless, this is not always the case, and will fail in particular 
for matrices whose variances are distributed over many different scales; 
in the latter case, the \emph{dimension-free} bound of Theorem 
\ref{thm:gine} can give rise to much sharper results.  Both Theorems 
\ref{thm:gine} and \ref{thm:bvh} therefore remain of significant 
independent interest.  Taken together, these results strongly support a 
fundamental conjecture, to be discussed in the next section, that would 
provide the ultimate understanding of the magnitude of the spectral norm 
of the random matrix model considered in this chapter.

The proof of Theorem \ref{thm:bvh} is completely different in nature 
than that of Theorem \ref{thm:gine}.  Rather than prove Theorem 
\ref{thm:bvh} in the general case, we will restrict attention in the 
rest of this section to the special case of \emph{sparse Wigner 
matrices}. The proof of Theorem \ref{thm:bvh} in the general case is 
actually no more difficult than in this special case, but the ideas and 
intuition behind the proof are particularly transparent when restricted 
to sparse Wigner matrices (which was how the authors of \cite{BvH16} 
arrived at the proof).  Once this special case has been understood, the 
reader can extend the proof to the general setting as an exercise, or 
refer to the general proof given in \cite{BvH16}.

\begin{example}[Sparse Wigner matrices]
\label{ex:sparse}
Informally, a sparse Wigner matrix is a symmetric random matrix with a 
given sparsity pattern, whose nonzero entries are independent standard 
Gaussian variables.  It is convenient to fix the sparsity pattern of the 
matrix by specifying a given undirected graph $G=([n],E)$ on $n$ 
vertices, whose adjacency matrix we denote as $B=(b_{ij})_{1\le 
i,j\le n}$. The corresponding sparse Wigner matrix $X$ is the symmetric 
random matrix whose entries are given by $X_{ij}=b_{ij}g_{ij}$, where 
$g_{ij}$ are independent standard Gaussian variables (up to symmetry 
$g_{ji}=g_{ij}$).  Clearly our previous Examples \ref{ex:diag},
\ref{ex:wigner}, and \ref{ex:block} are all special cases of this model.

For a sparse Wigner matrix, the first term
in Theorem \ref{thm:bvh} is precisely the maximal degree
$k=\mathrm{deg}(G)$ of the graph $G$, so that Theorem \ref{thm:bvh} reduces
to
$$
	\mathbf{E}\|X\| \lesssim \sqrt{k}+\sqrt{\log n}.
$$
We will see in section \ref{sec:open} that this bound is 
sharp for sparse Wigner matrices.
\end{example}

The remainder of this section is devoted to the proof of Theorem 
\ref{thm:bvh} in the setting of Example \ref{ex:sparse} (we fix
the notation introduced in this example in the sequel).
To understand the idea behind the proof, let us start by naively
writing out the central quantity that appears in moment method (Lemma \ref{lem:mommeth}):
we evidently have
\begin{align*}
	\mathbf{E}[\mathrm{Tr}[X^{2p}]] 
	&=
	\sum_{i_1,\ldots,i_{2p}=1}^n
	\mathbf{E}[X_{i_1i_2}X_{i_2i_3}\cdots X_{i_{2p-1}i_{2p}}
	X_{i_{2p}i_1}] \\
	&=
	\sum_{i_1,\ldots,i_{2p}=1}^n
	b_{i_1i_2}b_{i_2i_3}\cdots b_{i_{2p}i_1}\,
	\mathbf{E}[g_{i_1i_2}g_{i_2i_3}\cdots g_{i_{2p}i_1}].
\end{align*}
It is useful to think of $\gamma=(i_1,\ldots,i_{2p})$ geometrically
as a \emph{cycle}
$i_1\to i_2\to \cdots\to i_{2p}\to i_1$ of length $2p$.  The quantity
$b_{i_1i_2}b_{i_2i_3}\cdots b_{i_{2p}i_1}$ is equal to one precisely
when $\gamma$ defines a cycle in the graph $G$, and is zero otherwise.
We can therefore write
$$
	\mathbf{E}[\mathrm{Tr}[X^{2p}]] =
	\sum_{\mathrm{cycle}~\gamma~\mathrm{in}~G~\mathrm{of~length}~2p}
	c(\gamma),
$$
where we defined the constant
$c(\gamma):= \mathbf{E}[g_{i_1i_2}g_{i_2i_3}\cdots g_{i_{2p}i_1}]$.

It turns out that $c(\gamma)$ does not really depend on the the position 
of the cycle $\gamma$ in the graph $G$.  While we will not require a 
precise formula for $c(\gamma)$ in the proof, it is instructive 
to write down what it looks like.  For any cycle $\gamma$ in $G$, denote 
by $m_\ell(\gamma)$ the number of distinct edges in $G$ that are visited 
by $\gamma$ precisely $\ell$ times, and denote by 
$m(\gamma)=\sum_{\ell\ge 1}m_\ell(\gamma)$ the total number of distinct 
edges visited by $\gamma$.  Then 
$$
	c(\gamma) = \prod_{\ell=1}^\infty \mathbf{E}[g^\ell]^{m_\ell(\gamma)},
$$
where $g\sim N(0,1)$ is a standard Gaussian variable and we have used
the independence of the entries. From this formula, we read off two
important facts (which are the only ones that will actually be used 
in the proof):
\begin{itemize}
\item If any edge in $G$ is visited by $\gamma$ an odd number of times,
then $c(\gamma)=0$ (as the odd moments of $g$ vanish). Thus the only
cycles that matter are \emph{even} cycles, that is, cycles in which 
every distinct edge is visited an even number of times. 
\item $c(\gamma)$ depends on $\gamma$ only through the numbers
$m_\ell(\gamma)$.  Therefore, to compute $c(\gamma)$, we only need to
know the \emph{shape} $\mathbf{s}(\gamma)$ of the cycle $\gamma$.
\end{itemize}
The shape $\mathbf{s}(\gamma)$ is obtained from $\gamma$ by relabeling its
vertices in order of appearance; for example, the shape of the cycle
$7\to 3\to 9\to 7\to 3\to 9\to 7$ is given by
$1\to 2\to 3\to 1\to 2\to 3\to 1$.  The shape $\mathbf{s}(\gamma)$ captures the 
topological properties of $\gamma$ (such as the numbers $m_\ell(\gamma)=
m_\ell(\mathbf{s}(\gamma))$) without keeping track of the manner in which
$\gamma$ is embedded in $G$.  This is illustrated in the
following figure:
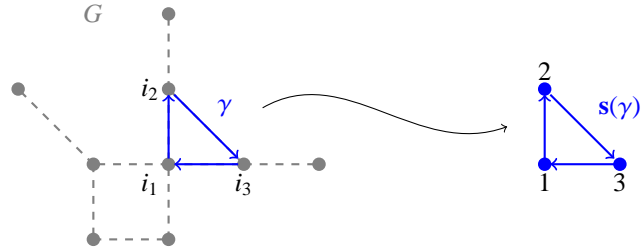
\begin{center}
\vskip.2cm
\begin{tikzpicture}
\draw[thick,gray,dashed] (0,0) to (1,0);
\draw[thick,gray,dashed] (1,0) to (0,1);
\draw[thick,gray,dashed] (0,1) to (0,0);
\draw[thick,gray,dashed] (0,0) to (-1,0);
\draw[thick,gray,dashed] (-1,0) to (-1,-1);
\draw[thick,gray,dashed] (-1,-1) to (0,-1);
\draw[thick,gray,dashed] (0,-1) to (0,0);
\draw[thick,gray,dashed] (1,0) to (2,0);
\draw[thick,gray,dashed] (0,1) to (0,2);
\draw[thick,gray,dashed] (-1,0) to (-2,1);
\draw[thick,gray,fill=gray] (0,0) circle (.075);
\draw[thick,gray,fill=gray] (1,0) circle (.075);
\draw[thick,gray,fill=gray] (0,1) circle (.075);
\draw[thick,gray,fill=gray] (-1,0) circle (.075);
\draw[thick,gray,fill=gray] (0,-1) circle (.075);
\draw[thick,gray,fill=gray] (-1,-1) circle (.075);
\draw[thick,gray,fill=gray] (0,2) circle (.075);
\draw[thick,gray,fill=gray] (-2,1) circle (.075);
\draw[thick,gray,fill=gray] (2,0) circle (.075);
\draw[thick,blue,->] (0,0.075) to (0,0.925);
\draw[thick,blue,->] (0.075,0.925) to (0.925,0.075);
\draw[thick,blue,->] (0.925,0) to (0.075,0);
\draw (0,0) node[below left] {$i_1$};
\draw (0,1) node[left] {$i_2$};
\draw (1,0) node[below] {$i_3$};
\draw[gray] (-1,2) node {$G$};
\draw[blue] (.75,.75) node {$\gamma$};

\draw[thick,blue,->] (5,0.075) to (5,0.925);
\draw[thick,blue,->] (5.075,0.925) to (5.925,0.075);
\draw[thick,blue,->] (5.925,0) to (5.075,0);
\draw[thick,blue,fill=blue] (5,0) circle (.075);
\draw[thick,blue,fill=blue] (6,0) circle (.075);
\draw[thick,blue,fill=blue] (5,1) circle (.075);
\draw (5,0) node[below] {$1$};
\draw (5,1) node[above] {$2$};
\draw (6,0) node[below] {$3$};
\draw[blue] (6,.75) node {$\mathbf{s}(\gamma)$};

\draw[->] (1.25,.75) to[out=30,in=200] (4.5,.5);
\end{tikzpicture}
\vskip.2cm
\end{center}
Putting together the above observations, we obtain the useful formula
$$
	\mathbf{E}[\mathrm{Tr}[X^{2p}]] =
	\sum_{\mathrm{shape}~\mathbf{s}~\mathrm{of~even~cycle~of~length}~2p}
	c(\mathbf{s})\times
	\#\{\mbox{embeddings of }\mathbf{s}\mbox{ in }G\}.
$$
So far, we have done nothing but bookkeeping. To use the above bound, 
however, we must get down to work and count the number of shapes of even 
cycles that can appear in the given graph $G$. The problem we face is 
that the latter proves to be a difficult combinatorial problem,
which is apparently completely intractable when presented with any given
graph $G$ that may possess an arbitrary structure (this is already highly 
nontrivial even in a complete graph when $p$ is large!) To squeeze anything
useful out of this bound, it is essential that we find a shortcut.

The solution to our problem proves to be incredibly simple. Recall that 
$G$ is a given graph of degree $\mathrm{deg}(G)=k$.  Of all graphs of 
degree $k$, which one will admit the most possible shapes? Obviously the 
graph that admits the most shapes is the one where every potential edge 
between two vertices is present; therefore, the graph of degree $k$ that 
possesses the most shapes is the \emph{complete graph on $k$ vertices}. 
From the random matrix point of view, the latter corrsponds to a Wigner 
matrix of dimension $k\times k$.  This simple idea suggests that rather 
than directly estimating the quantity $\mathbf{E}[\mathrm{Tr}[X^{2p}]]$ 
by combinatorial means, we should aim to prove a \emph{comparison 
principle} between the moments of the $n\times n$ sparse matrix $X$ and 
the moments of a $k\times k$ Wigner matrix $Y$, which we already know 
how to bound by Lemma \ref{lem:wigner}. Note that such a comparison 
principle is of a completely different nature than the Slepian-Fernique 
method used previously: here we are comparing two matrices 
of \emph{different dimension}. The intuitive idea is that a large sparse 
matrix can be ``compressed'' into a much lower dimensional dense matrix 
without decreasing its norm.

The alert reader will note that there is a problem with the above 
intuition. While the complete graph on $k$ points admits more shapes
than the original graph $G$, there are less potential ways in which each
shape can be embedded in the complete graph as the latter possesses less
vertices than the original graph. We can compensate for this
deficiency by slightly increasing the dimension of the complete graph.

\begin{lemma}[Dimension compression]
\label{lem:compres}
Let $X$ be the $n\times n$ sparse Wigner matrix (Example \ref{ex:sparse})
defined by a graph $G=([n],E)$ of maximal degree $\mathrm{deg}(G)=k$,
and let $Y_r$ be an $r\times r$ Wigner matrix (Example \ref{ex:wigner}).
Then, for every $p\ge 1$,
$$
	\mathbf{E}[\mathrm{Tr}[X^{2p}]] \le
	\frac{n}{k+p}\,\mathbf{E}[\mathrm{Tr}[Y_{k+p}^{2p}]].
$$
\end{lemma}

\begin{proof}
Let $\mathbf{s}$ be the shape of an even cycle of length $2p$, and let 
$K_r$ be the complete graph on $r>p$ points.  Denote by $m(\mathbf{s})$ 
the number of distinct vertices in $\mathbf{s}$, and note that 
$m(\mathbf{s})\le p+1$ as every distinct edge in $\mathbf{s}$ must 
appear at least twice.  Thus
$$
	\#\{\mbox{embeddings of }\mathbf{s}\mbox{ in }K_r\} =
	r(r-1)\cdots(r-m(\mathbf{s})+1),
$$
as any assignment of
vertices of $K_r$ to the distinct vertices of $\mathbf{s}$ defines
a valid embedding of $\mathbf{s}$ in the complete graph.
On the other hand, to count the number of embeddings of
$\mathbf{s}$ in $G$, note that we have as many as $n$ choices for
the first vertex, while each subsequent vertex can be chosen in at
most $k$ ways (as $\mathrm{deg}(G)=k$).  Thus
$$
	\#\{\mbox{embeddings of }\mathbf{s}\mbox{ in }G\} \le
	nk^{m(\mathbf{s})-1}.
$$
Therefore, if we choose $r=k+p$, we have
$r-m(\mathbf{s})+1 \ge r-p \ge k$, so that
$$
	\#\{\mbox{embeddings of }\mathbf{s}\mbox{ in }G\} \le
	\frac{n}{r}\,
	\#\{\mbox{embeddings of }\mathbf{s}\mbox{ in }K_r\}.
$$
The proof now follows from the combinatorial expression for $\mathbf{E}[\mathrm{Tr}[X^{2p}]]$.
\qed\end{proof}

With Lemma \ref{lem:compres} in hand, it is now straightforward to complete
the proof of Theorem \ref{thm:bvh} for the sparse Wigner matrix model of
Example \ref{ex:sparse}.

\begin{proof}[Proof of Theorem \ref{thm:bvh} in the setting of Example
\ref{ex:sparse}]
We begin by noting that
$$
	\mathbf{E}\|X\| 
	\le
	\mathbf{E}[\|X\|^{2p}]^{1/2p}
	\le
	n^{1/2p}\, 
	\mathbf{E}[\|Y_{k+p}\|^{2p}]^{1/2p}
$$
by Lemma \ref{lem:compres}, where we used
$\|X\|^{2p}\le\mathrm{Tr}[X^{2p}]$ and $\mathrm{Tr}[Y_r^{2p}] \le
r\|Y_r\|^{2p}$.
Thus
\begin{align*}
	\mathbf{E}\|X\| &\lesssim
	\mathbf{E}[\|Y_{k+\lfloor\log n\rfloor}\|^{2\log n}]^{1/2\log n} \\
	&\le
	\mathbf{E}\|Y_{k+\lfloor\log n\rfloor}\| +
	\mathbf{E}[(\|Y_{k+\lfloor\log n\rfloor}\|-
	\mathbf{E}\|Y_{k+\lfloor\log n\rfloor}\|)^{2\log n}]^{1/2\log n}
	\\
	&\lesssim
	\sqrt{k+\log n} + \sqrt{\log n},
\end{align*}
where in the last inequality we used Lemma \ref{lem:wigner} to bound 
the first term and Lemma \ref{lem:gaussconc} to bound the second term.
Thus $\mathbf{E}\|X\|\lesssim\sqrt{k}+\sqrt{\log n}$, completing
the proof.
\qed\end{proof}

\subsection{Three conjectures}
\label{sec:open}

We have obtained in the previous sections two remarkably sharp bounds on 
the spectral norm of random matrices with independent centered Gaussian 
entries: the slightly suboptimal dimension-free bound of Theorem 
\ref{thm:gine} for $\mathbf{E}\|X\|$, and the sharp dimension-dependent 
bound of Theorem \ref{thm:bvh} for $\mathbf{E}[\|X\|^{\log n}]^{1/\log 
n}$.  As we will shortly argue, the latter bound is also sharp for 
$\mathbf{E}\|X\|$ in almost all situations of practical interest. 
Nonetheless, we cannot claim to have a complete understanding of the 
mechanisms that control the spectral norm of Gaussian random matrices 
unless we can obtain a sharp dimension-free bound on $\mathbf{E}\|X\|$.
While this problem remains open, the above results strongly suggest
what such a sharp bound should look like.

To gain some initial intuition, let us complement the sharp lower bound of
Corollary \ref{cor:bvhlower} for $\mathbf{E}[\|X\|^{\log n}]^{1/\log n}$
by a trivial lower bound for $\mathbf{E}\|X\|$.

\begin{lemma}
\label{lem:sharplower}
In the setting of this section, we have
$$
	\mathbf{E}\|X\|\gtrsim
	\max_{i\le n}\sqrt{\sum_{j=1}^n b_{ij}^2} +
	\mathbf{E}\max_{i,j\le n}|X_{ij}|.
$$
\end{lemma}

\begin{proof}
The first term is a lower bound by Lemma \ref{lem:lowernck}, while the
second term is a lower bound by the trivial pointwise inequality
$\|X\|\ge \max_{i,j}|X_{ij}|$.
\qed\end{proof}

The simplest possible upper bound on the maximum of centered Gaussian 
random variables is $\mathbf{E}\max_{i,j}|X_{ij}|\lesssim 
\max_{i,j}b_{ij}\sqrt{\log n}$, which is sharp for i.i.d.\ Gaussian 
variables.  Thus the lower bound of Lemma \ref{lem:sharplower} matches 
the upper bound of Theorem \ref{thm:bvh} under a minimal homogeneity 
assumption: it suffices to assume that the number of entries whose 
standard deviation $b_{kl}$ is of the same order as $\max_{i,j}b_{ij}$ 
grows polynomially with dimension (which still allows for a vanishing 
fraction of entries of the matrix to possess large variance). For 
example, in the sparse Wigner matrix model of Example \ref{ex:sparse}, 
every row of the matrix that does not correspond to an isolated vertex 
in $G$ contains at least one entry of variance one. Therefore, if $G$ 
possesses no isolated vertices, there are at least $n$ entries of $X$
with variance one, and it follows immediately from Lemma \ref{lem:sharplower}
that the bound of Theorem \ref{thm:bvh} is sharp for sparse Wigner matrices.
(There is no loss of generality in assuming that $G$ has no isolated vertices:
any isolated vertex yields a row that is identically zero, so we can simply
remove such vertices from the graph without changing the norm.)

However, when the variances of the entries of $X$ possess many different scales,
the dimension-dependent upper bound $\mathbf{E}\max_{i,j}|X_{ij}|\lesssim
\max_{i,j}b_{ij}\sqrt{\log n}$ can fail to be sharp. To obtain a sharp
bound on the maximum of Gaussian random variables, we must proceed in a
dimension-free fashion as in Lemma \ref{lem:gaussmax}.  In particular,
combining Remark \ref{rem:gaussmax} and Lemma \ref{lem:sharplower}
yields the following explicit lower bound:
$$
	\mathbf{E}\|X\|\gtrsim
	\max_{i\le n}\sqrt{\sum_{j=1}^n b_{ij}^2} +
	\max_{i,j\le n} b_{ij}\sqrt{\log i},
$$
provided that $\max_jb_{1j}\ge \max_jb_{2j} \ge\cdots\ge
\max_jb_{nj}>0$ (there is no loss of generality in assuming the latter,
as we can always permute the rows and columns of $X$ to achieve this
ordering without changing the norm of $X$).
It will not have escaped the attention of
the reader that the latter lower bound is tantalizingly close both
to the dimension-dependent upper bound of Theorem \ref{thm:bvh},
and to the dimension-free upper bound of Theorem \ref{thm:gine}.
This leads us to the following very natural conjecture \cite{vH15}.

\begin{conj}
\label{conj:1}
Assume without loss of generality that the rows and columns of $X$ have
been permuted such that $\max_jb_{1j}\ge \max_jb_{2j} \ge\cdots\ge
\max_jb_{nj}>0$.  Then
\begin{align*}
	\mathbf{E}\|X\| &\asymp
	\|\mathbf{E}X^2\|^{1/2} + \mathbf{E}\max_{i,j\le n}|X_{ij}| \\
	&\asymp
	\max_{i\le n}\sqrt{\sum_{j=1}^n b_{ij}^2} +
	\max_{i,j\le n} b_{ij}\sqrt{\log i}.
\end{align*}
\end{conj}

Conjecture \ref{conj:1} appears completely naturally from our results, 
and has a surprising interpretation.  There are two simple mechanisms 
that would certainly force the random matrix $X$ to have large expected 
norm $\mathbf{E}\|X\|$: the matrix $X$ is can be large ``on average'' in 
the sense that $\|\mathbf{E}X^2\|$ is large (note that the expectation 
here is \emph{inside} the norm), or the matrix $X$ can have an entry 
that exhibits a large fluctuation in the sense that $\max_{i,j}X_{ij}$ 
is large.  Conjecture \ref{conj:1} suggests that these two mechanisms 
are, in a sense, the \emph{only} reasons why $\mathbf{E}\|X\|$ can be 
large. 

Given the remarkable similarity between Conjecture \ref{conj:1} and 
Theorem \ref{thm:bvh}, one might hope that a slight sharpening of the 
proof of Theorem \ref{thm:bvh} would suffice to yield the conjecture. 
Unfortunately, it seems that the moment method is largely useless for 
the purpose of obtaining dimension-free bounds: indeed, the Corollary 
\ref{cor:bvhlower} shows that the moment method is already exploited 
optimally in the proof of Theorem \ref{thm:bvh}.  While it is sometimes 
possible to derive dimension-free results from dimension-dependent 
results by a stratification procedure, such methods either fail 
completely to capture the correct structure of the problem (cf.\ 
\cite{RS13}) or retain a residual dimension-dependence (cf.\ 
\cite{vH15}).  It therefore seems likely that random process methods
will prove to be essential for progress in this direction.

While Conjecture \ref{conj:1} appears completely natural in the present 
setting, we should also discuss a competing conjecture that was proposed 
much earlier by R.\ Lata{\l}a.  Inspired by certain results of Seginer 
\cite{Seg00} for matrices with i.i.d.\ entries, Lata{\l}a conjectured 
the following sharp bound in the general setting of this section.

\begin{conj}
\label{conj:2}
In the setting of this section, we have
$$
	\mathbf{E}\|X\|\asymp
	\mathbf{E}\max_{i\le n}\sqrt{\sum_{j=1}^nX_{ij}^2}.
$$
\end{conj}

As $\|X\|^2\ge \max_i\sum_j X_{ij}^2$ holds deterministically, the lower
bound in Conjecture \ref{conj:2} is trivial: it states that a matrix
that possesses a large row must have large spectral norm.  Conjecture
\ref{conj:2} suggests that this is the \emph{only} reason why the
matrix norm can be large.  This is certainly not the case for an
arbitrary matrix $X$, and so it is not at all clear \emph{a priori} why
this should be true. Nonetheless, no counterexample is known in the
setting of the Gaussian random matrices considered in this section.

While Conjectures \ref{conj:1} and \ref{conj:2} appear to arise from 
different mechanisms, it is observed in \cite{vH15} that these 
conjectures are actually equivalent: it is not difficult to show that 
the right-hand side in both inequalities is equivalent, up to the 
universal constant, to the explicit expression recorded in Conjecture 
\ref{conj:1}.  In fact, let us note that both conjectured mechanisms are 
essentially already present in the proof of Theorem \ref{thm:gine}: in 
the comparison process $Y_v$ that arises in the proof, the first term is 
strongly reminiscent of Conjecture \ref{conj:2}, while the second term 
is reminiscent of the second term in Conjecture \ref{conj:1}.  In this 
sense, the mechanism that is developed in the proof of Theorem 
\ref{thm:gine} provides even stronger evidence for the validity of these 
conjectures.  The remaining inefficiency in the proof of Theorem 
\ref{thm:gine} is discussed in detail in \cite{vH15}.

We conclude by discussing briefly a much more speculative question. The 
noncommutative Khintchine inequalities developed in the previous section 
hold in a very general setting, but are almost always suboptimal. In 
contrast, the bounds in this section yield nearly optimal results under 
the additional assumption that the matrix entries are independent. It 
would be very interesting to understand whether the bounds of the 
present section can be extended to the much more general setting 
captured by noncommutative Khintchine inequalities.  Unfortunately, 
independence is used crucially in the proofs of the results in this 
section, and it is far from clear what mechanism might give rise to 
analogous results in the dependent setting.

One might nonetheless speculate what such a result might potentially 
look like. In particular, we note that both parameters 
that appear in the sharp bound Theorem \ref{thm:bvh} have natural 
analogues in the general setting: in the setting of this section
$$
	\|\mathbf{E}X^2\| =
	\sup_{v\in B}\mathbf{E}\langle v,X^2v\rangle =
	\max_i\sum_jb_{ij}^2,
	\qquad\quad
	\sup_{v\in B}\mathbf{E}\langle v,Xv\rangle^2 =
	\max_{i,j}b_{ij}^2.
$$
We have already encountered both these quantities also in the previous
section: $\sigma=\|\mathbf{E}X^2\|^{1/2}$ is the natural structural parameter that arises in
noncommutative Khintchine inequalities, while $\sigma_*:=\sup_v\mathbf{E}[
\langle v,Xv\rangle^2]^{1/2}$ controls the fluctuations of the spectral norm
by Gaussian concentration (see the proof of Lemma \ref{lem:lowernck}). 
By analogy with Theorem \ref{thm:bvh}, we might therefore speculatively
conjecture:

\begin{conj}
\label{conj:3}
Let $X=\sum_{k=1}^s g_kA_k$ as in Theorem \ref{thm:nck}. Then
$$
	\mathbf{E}\|X\| \lesssim
	\|\mathbf{E}X^2\|^{1/2} + 
	\sup_{v\in B}\mathbf{E}[\langle v,Xv\rangle^2]^{1/2}\sqrt{\log n}.
$$
\end{conj}

Such a generalization would constitute a far-reaching improvement of the 
noncommutative Khintchine theory. The problem with Conjecture 
\ref{conj:3} is that it is completely unclear how such a bound might 
arise: the only evidence to date for the potential validity of such a 
bound is the vague analogy with the independent case, and the fact that 
a counterexample has yet to be found.

\subsection{Seginer's inequality}
\label{sec:seginer}

Throughout this chapter, we have focused attention on Gaussian random 
matrices. We depart briefly from this setting in this section to discuss 
some aspects of structured random matrices that arise under other 
distributions of the entries.

The main reason that we restricted attention to Gaussian matrices is 
that most of the difficulty of capturing the structure of the matrix 
arises in this setting; at the same time, all upper bounds we develop 
extend without difficulty to more general distributions, so there is no 
significant loss of generality in focusing on the Gaussian case. For 
example, let us illustrate the latter statement using the moment method.

\begin{lemma}
\label{lem:univ}
Let $X$ and $Y$ be symmetric random matrices with independent entries
(modulo symmetry).  Assume that $X_{ij}$ are centered and subgaussian, 
that is, $\mathbf{E}X_{ij}=0$ and $\mathbf{E}[X_{ij}^{2p}]^{1/2p}
\lesssim b_{ij}\sqrt{p}$ for all $p\ge 1$, and let
$Y_{ij}\sim N(0,b_{ij}^2)$.  Then
$$
	\mathbf{E}[\mathrm{Tr}[X^{2p}]]^{1/2p}
	\lesssim
	\mathbf{E}[\mathrm{Tr}[Y^{2p}]]^{1/2p}
	\quad\mbox{for all }p\ge 1.
$$
\end{lemma}

\begin{proof}
Let $X'$ be an independent copy of $X$. Then
$\mathbf{E}[\mathrm{Tr}[X^{2p}]] =
\mathbf{E}[\mathrm{Tr}[(X-\mathbf{E}X')^{2p}]] \le
\mathbf{E}[\mathrm{Tr}[(X-X')^{2p}]]$ by Jensen's
inequality.  Moreover, $Z=X-X'$ a symmetric random
matrix satisfying the same properties as $X$, with the additional
property that the entries $Z_{ij}$ have symmetric distribution.
Thus $\mathbf{E}[Z_{ij}^p]^{1/p} \lesssim \mathbf{E}[Y_{ij}^p]^{1/p}$
for all $p\ge 1$ (for odd $p$ both sides are zero by symmetry, while
for even $p$ this follows from the subgaussian assumption using
$\mathbf{E}[Y_{ij}^{2p}]^{1/2p}\asymp b_{ij}\sqrt{p}$).  It remains to
note that
\begin{align*}
	\mathbf{E}[\mathrm{Tr}[X^{2p}]] &=
	\sum_{\mathrm{cycle}~\gamma~\mathrm{of~length}~2p}
	\prod_{1\le i\le j\le n}\mathbf{E}[X_{ij}^{\#_{ij}(\gamma)}]
	\\ &\le
	C^{2p}
	\sum_{\mathrm{cycle}~\gamma~\mathrm{of~length}~2p}
	\prod_{1\le i\le j\le n}\mathbf{E}[Y_{ij}^{\#_{ij}(\gamma)}]
	=
	C^{2p}\,\mathbf{E}[\mathrm{Tr}[Y^{2p}]]
\end{align*}
for a universal constant $C$, where $\#_{ij}(\gamma)$ denotes the
number of times the edge $(i,j)$ appears in the cycle $\gamma$.
The conclusion follows immediately.
\qed\end{proof}

Lemma \ref{lem:univ} shows that to upper bound the moments of a 
subgaussian random matrix with independent entries, it suffices to 
obtain a bound in the Gaussian case. The reader may readily verify that 
the completely analogous approach can be applied in the more general 
setting of the noncommutative Khinthchine inequality. On the other hand, 
Gaussian bounds using the random process method extend to the 
subgaussian setting by virtue of a general subgaussian comparison 
principle \cite[Theorem 2.4.12]{Tal14}. Beyond the subgaussian setting, 
similar methods can be used for entries with heavy-tailed distributions, 
see for example \cite{BvH16}.

The above observations indicate that, in some sense, Gaussian random 
matrices are the ``worst case'' among subgaussian matrices. One can go 
one step further and ask whether there is some form of universality: do 
all subgaussian random matrices behave like their Gaussian counterparts?  
The universality phenomenon plays a major role in recent advances in 
random matrix theory: it turns out that many properties of Wigner 
matrices do not depend on the distribution of the entries.  
Unfortunately, we cannot expect universal behavior for structured random 
matrices: while Gaussian matrices are the ``worst case'' among 
subgaussian matrices, matrices with subgaussian entries can sometimes 
behave much better. The simplest example is the case of diagonal 
matrices (Example \ref{ex:diag}) with i.i.d.\ entries on the diagonal: 
in the Gaussian case $\mathbf{E}\|X\|\asymp \sqrt{\log n}$, but 
obviously $\mathbf{E}\|X\|\asymp 1$ if the entries are uniformly bounded 
(despite that uniformly bounded random variables are obviously 
subgaussian). In view of such examples, there is little hope to obtain a 
complete understanding of structured random matrices for arbitrary 
distributions of the entries.  This justifies the approach we have taken:
we seek sharp bounds for Gaussian matrices, which give rise to powerful
upper bounds for general distributions of the entries.

\begin{remark}
We emphasize in this context that Conjectures \ref{conj:1} and 
\ref{conj:2} in the previous section are fundamentally Gaussian 
in nature, and \emph{cannot} hold as stated for subgaussian matrices.
For a counterexample along the lines of Example \ref{ex:block}, see
\cite{Seg00}.
\end{remark}

Despite these negative observations, it can be of significant interest 
to go beyond the Gaussian setting to understand whether the bounds we 
have obtained can be systematically improved under more favorable 
assumptions on the distributions of the entries. To illustrate how such
improvements could arise, we discuss a result of Seginer \cite{Seg00}
for random matrices with independent \emph{uniformly bounded} entries.

\begin{theorem}[Seginer]
\label{thm:seginer}
Let $X$ be an $n\times n$ symmetric random matrix with independent entries
(modulo symmetry) and $\mathbf{E}X_{ij}=0$, 
$\|X_{ij}\|_\infty \lesssim b_{ij}$ for all $i,j$. Then
$$
	\mathbf{E}\|X\| \lesssim
	\max_{i\le n}
	\sqrt{\sum_{j=1}^n b_{ij}^2}
	~(\log n)^{1/4}.
$$
\end{theorem}

The uniform bound $\|X_{ij}\|_\infty\lesssim b_{ij}$ certainly implies 
the much weaker subgaussian property 
$\mathbf{E}[X_{ij}^{2p}]^{1/2p}\lesssim b_{ij}\sqrt{p}$, so that the 
conclusion of Theorem \ref{thm:bvh} extends immediately to the present 
setting by Lemma \ref{lem:univ}. In many cases, the latter bound is much 
sharper than the one provided by Theorem \ref{thm:seginer}; indeed, 
Theorem \ref{thm:seginer} is suboptimal even for Wigner matrices (it 
could be viewed of a variant of the noncommutative Khintchine inequality 
in the present setting with a smaller power in the logarithmic factor). 
However, the interest of Theorem \ref{thm:seginer} is that it 
\emph{cannot} hold for Gaussian entries: for example, in the diagonal 
case $b_{ij}=\mathbf{1}_{i=j}$, Theorem \ref{thm:seginer} gives 
$\mathbf{E}\|X\|\lesssim (\log n)^{1/4}$ while any Gaussian bound must 
give at least $\mathbf{E}\|X\|\gtrsim\sqrt{\log n}$.  In this sense, 
Theorem \ref{thm:seginer} illustrates that it is possible in some cases
to exploit the effect of stronger distributional assumptions in order to 
obtain improved bounds for non-Gaussian random matrices. The simple 
proof that we will give (taken from \cite{BvH16}) shows very clearly how 
this additional distributional information enters the picture.

\begin{proof}[Proof of Theorem \ref{thm:seginer}]
The proof works by combining two very different bounds on the matrix
norm.  On the one hand, due to Lemma \ref{lem:univ}, we can directly apply
the Gaussian bound of Theorem \ref{thm:bvh} in the present setting.
On the other hand, as the entries of $X$ are uniformly bounded,
we can do something that is impossible for Gaussian random variables:
we can \emph{uniformly} bound the norm $\|X\|$ as
\begin{align*}
	\|X\| &=
	\sup_{v\in B} \Bigg|\sum_{i,j=1}^n v_i X_{ij} v_j\Bigg|
	\le
	\sup_{v\in B} \sum_{i,j=1}^n (|v_i|\,|X_{ij}|^{1/2})
	(|X_{ij}|^{1/2}|v_j|) \\
	&\le
	\sup_{v\in B} 
	\sum_{i,j=1}^n v_i^2|X_{ij}| =
	\max_{i\le n}\sum_{j=1}^n |X_{ij}|
	\le
	\max_{i\le n}\sum_{j=1}^n b_{ij},
\end{align*}
where we have used the Cauchy-Schwarz inequality in going from the first 
to the second line.  The idea behind the proof of Theorem 
\ref{thm:seginer} is roughly as follows.  Many small entries of $X$ can 
add up to give rise to a large norm; we might expect the 
cumulative effect of many independent centered random variables to give 
rise to Gaussian behavior.  On the other hand, if a few large entries of 
$X$ dominate the norm, there is no Gaussian behavior and we expect
that the uniform bound provides much better control.
To capture this idea, we partition the matrix into two parts
$X=X_1+X_2$, where $X_1$ contains the ``small'' entries and $X_2$
contains the ``large'' entries:
$$
	(X_1)_{ij} = X_{ij}\mathbf{1}_{b_{ij}\le u},\qquad\quad
	(X_2)_{ij} = X_{ij}\mathbf{1}_{b_{ij}>u}.
$$
Applying the 
Gaussian bound to $X_1$ and the uniform bound to $X_2$ yields
\begin{align*}
	\mathbf{E}\|X\| &\le
	\mathbf{E}\|X_1\| + \mathbf{E}\|X_2\| 
	\\ &\lesssim
	\max_{i\le n}\sqrt{\sum_{j=1}^n b_{ij}^2\mathbf{1}_{b_{ij}\le u}}
	+ u\sqrt{\log n}
	+
	\max_{i\le n}\sum_{j=1}^n b_{ij}\mathbf{1}_{b_{ij}>u}
	\\ &\le
	\max_{i\le n}\sqrt{\sum_{j=1}^n b_{ij}^2}
	+ u\sqrt{\log n}
	+ \frac{1}{u}\max_{i\le n}\sum_{j=1}^n b_{ij}^2.
\end{align*}
The proof is completed by optimizing over $u$.
\qed\end{proof}

The proof of Theorem \ref{thm:seginer} illustrates the improvement that 
can be achived by trading off between Gaussian and uniform bounds on the 
norm of a random matrix. Such tradeoffs play a fundamental role in the 
general theory that governs the suprema of bounded random processes 
\cite[Chapter 5]{Tal14}. Unfortunately, this tradeoff is captured only 
very crudely by the suboptimal Theorem \ref{thm:seginer}.

Developing a sharp understanding of the behavior of bounded random 
matrices is a problem of significant interest: the bounded analogue of 
sparse Wigner matrices (Example \ref{ex:sparse}) has interesting 
connections with graph theory and computer science, cf.\ \cite{AKM13} 
for a review of such applications.  Unlike in the Gaussian case, 
however, it is clear that the degree of the graph that defines a sparse 
Wigner matrix cannot fully explain its spectral norm in the present 
setting: very different behavior is exhibited in dense vs.\ 
locally tree-like graphs of the same degree \cite[section 
4.2]{BvH16}. To date, a deeper understanding of such matrices beyond 
the Gaussian case remains limited.

\section{Sample covariance matrices}
\label{sec:sampcov}

We finally turn our attention to a random matrix model that is somewhat 
different than the matrices we considered so far.  The following model 
will be considered throughout this section.  Let $\Sigma$ be a given 
$d\times d$ positive semidefinite matrix, and let $X_1,X_2,\ldots,X_n$ 
be i.i.d.\ centered Gaussian random vectors in $\mathbb{R}^d$ with 
covariance matrix $\Sigma$.  We consider in the following the 
$d\times d$ symmetric random matrix
$$
	Z = \frac{1}{n}\sum_{k=1}^n X_kX_k^* = \frac{XX^*}{n},
$$
where we defined the $d\times n$ matrix $X_{ik}=(X_k)_i$. In contrast to 
the models considered in the previous sections, the random matrix $Z$ is 
not centered: we have in fact $\mathbf{E}Z=\Sigma$. This gives rise to 
the classical statistical interpretation of this matrix. We can think of 
$X_1,\ldots,X_n$ as being i.i.d.\ data drawn from a centered Gaussian 
distribution with unknown covariance matrix $\Sigma$. In this setting, 
the random matrix $Z$, which depends only on the observed data, provides 
an unbiased estimator of the covariance matrix of the underlying data. 
For this reason, $Z$ is known as the \emph{sample covariance matrix}. Of 
primary interest in this setting is not so much the matrix norm 
$\|Z\|=\|X\|^2/n$ itself, but rather the deviation $\|Z-\Sigma\|$ of $Z$ 
from its mean.

The model of this section could be viewed as being 
``semi-structured.'' On the one hand, the covariance matrix $\Sigma$ is 
completely arbitrary, and it therefore allows for an arbitrary variance 
and dependence pattern within each column of the matrix $X$ (as in the 
most general setting of the noncommutative Khintchine inequality). On 
the other hand, the columns of $X$ are assumed to be i.i.d., 
so that no nontrivial structure among the columns is captured by the 
present model. While the latter assumption is limiting, it allows us
to obtain a complete understanding of the structural parameters that
control the expected deviation $\mathbf{E}\|Z-\Sigma\|$ in this setting
\cite{KL16}.

\begin{theorem}[Koltchinskii-Lounici]
\label{thm:kolt}
In the setting of this section
$$
	\mathbf{E}\|Z-\Sigma\| \asymp
	\|\Sigma\|\,\Bigg(
	\sqrt{\frac{r(\Sigma)}{n}} + \frac{r(\Sigma)}{n} 
	\Bigg),
$$
where $r(\Sigma) := \mathrm{Tr}[\Sigma]/\|\Sigma\|$ is the
\emph{effective rank} of $\Sigma$.
\end{theorem}

The remainder of this section is devoted to the proof of Theorem
\ref{thm:kolt}.

\subsubsection*{Upper bound}

The proof of Theorem \ref{thm:kolt} will use the random process method 
using tools that were already developed in the previous sections. 
It would be clear how to proceed if we wanted to bound $\|Z\|$:
as $\|Z\|=\|X\|^2/n$, it would suffice to bound
$\|X\|$ which is the supremum of a Gaussian process.  Unfortunately,
this idea does not extend directly to the problem of bounding
$\|Z-\Sigma\|$: the latter quantity is not the supremum of a centered 
Gaussian process, but rather of a \emph{squared} Gaussian process
$$
	\|Z-\Sigma\| =
	\sup_{v\in B}\Bigg|\frac{1}{n}\sum_{k=1}^n
	\{\langle v,X_k\rangle^2-\mathbf{E}\langle v,X_k\rangle^2\}
	\Bigg|.
$$
We therefore cannot directly apply a Gaussian comparison method such as
the Slepian-Fernique inequality to control the expected deviation
$\mathbf{E}\|Z-\Sigma\|$.

To surmount this problem, we will use a simple device that is widely 
used in the study of squared Gaussian processes (or \emph{Gaussian 
chaos}), cf.\ \cite[section 3.2]{LT91}.

\begin{lemma}[Decoupling]
\label{lem:decouple}
Let $\tilde X$ be an independent copy of $X$. Then
$$
	\mathbf{E}\|Z-\Sigma\| \le \frac{2}{n}\,
	\mathbf{E}\|X\tilde X^*\|.
$$
\end{lemma}

\begin{proof}
By Jensen's inequality
$$
	\mathbf{E}\|Z-\Sigma\| 
	=
	\frac{1}{n}\,\mathbf{E}\|\mathbf{E}[(X+\tilde X)(X-\tilde X)^*|X]\|
	\le
	\frac{1}{n}\,\mathbf{E}\|(X+\tilde X)(X-\tilde X)^*\|.
$$
It remains to note that $(X+\tilde X,X-\tilde X)$ has the same distribution
as $\sqrt{2}\,(X,\tilde X)$.
\qed\end{proof}

Roughly speaking, the decoupling device of Lemma \ref{lem:decouple} 
allows us to replace the square $XX^*$ of a Gaussian matrix by a product 
of two independent copies $X\tilde X^*$.  While the latter is still
not Gaussian, it becomes Gaussian if we condition on one of the copies
(say, $\tilde X$).  This means that $\|X\tilde X^*\|$ is the supremum
of a Gaussian process \emph{conditionally} on $\tilde X$.  This is precisely
what we will exploit in the sequel: we use the Slepian-Fernique inequality
conditionally on $\tilde X$ to obtain the following bound.

\begin{lemma}
\label{lem:sqreduce}
In the setting of this section
$$
	\mathbf{E}\|Z-\Sigma\| \lesssim
	\mathbf{E}\|X\|\frac{\sqrt{\mathrm{Tr}[\Sigma]}}{n} +
	\|\Sigma\|\sqrt{\frac{r(\Sigma)}{n}}.
$$
\end{lemma}

\begin{proof}
By Lemma \ref{lem:decouple} we have
$$
	\mathbf{E}\|Z-\Sigma\| \le \frac{2}{n}\,
	\mathbf{E}\bigg[\sup_{v,w\in B}Z_{v,w}\bigg],\qquad\quad
	Z_{v,w} := 
	\sum_{k=1}^n \langle v,X_k\rangle \langle w,\tilde X_k\rangle.
$$
Writing for simplicity $\mathbf{E}_{\tilde X}[\cdot] =
\mathbf{E}[\cdot|\tilde X]$, we can estimate
\begin{align*}
	\mathbf{E}_{\tilde X}(Z_{v,w}-Z_{v',w'})^2 &\le
	2\langle v-v',\Sigma(v-v')\rangle 
	\sum_{k=1}^n \langle w,\tilde X_k\rangle^2
	+
	2\langle v',\Sigma v'\rangle 
	\sum_{k=1}^n 
	\langle w-w',\tilde X_k\rangle^2 \\
	& \le
	2 \|\tilde X\|^2 \|\Sigma^{1/2}(v-v')\|^2 +
	2\|\Sigma\|\,\|\tilde X^*(w-w')\|^2 \\
	&=
	\mathbf{E}_{\tilde X}(Y_{v,w}-Y_{v',w'})^2,
	\phantom{\sum_{k=1}^n}
\end{align*}
where we defined
$$
	Y_{v,w} = 
	\sqrt{2}\,\|\tilde X\|\,\langle v,\Sigma^{1/2}g\rangle +
	(2\|\Sigma\|)^{1/2}\,\langle w,\tilde X g'\rangle
$$
with $g,g'$ independent standard Gaussian vectors in $\mathbb{R}^d$
and $\mathbb{R}^n$, respectively. Thus
\begin{align*}
	\mathbf{E}_{\tilde X}\bigg[\sup_{v,w\in B}Z_{v,w}\bigg]
	\le
	\mathbf{E}_{\tilde X}\bigg[\sup_{v,w\in B}Y_{v,w}\bigg]
	&\lesssim
	\|\tilde X\|\, \mathbf{E}\|\Sigma^{1/2}g\| +
	\|\Sigma\|^{1/2}\,\mathbf{E}_{\tilde X}\|\tilde X g\|
	\\
	&\le
	\|\tilde X\|\sqrt{\mathrm{Tr}[\Sigma]} +
	\|\Sigma\|^{1/2}\,\mathrm{Tr}[\tilde X \tilde X^*]^{1/2}
\end{align*}
by the Slepian-Fernique inequality.  Taking the expectation with respect
to $\tilde X$ and using
that $\mathbf{E}\|\tilde X\|=\mathbf{E}\|X\|$ and
$\mathbf{E}[\mathrm{Tr}[\tilde X \tilde X^*]^{1/2}] \le
\sqrt{n\,\mathrm{Tr}[\Sigma]}$ yields the conclusion.
\qed\end{proof}

Lemma \ref{lem:sqreduce} has reduced the problem of bounding 
$\mathbf{E}\|Z-\Sigma\|$ to the much more straightforward problem of 
bounding $\mathbf{E}\|X\|$: as $\|X\|$ is the supremum of a Gaussian 
process, the latter is amenable to a direct application of the 
Slepian-Fernique inequality precisely as was done in the proof of Lemma 
\ref{lem:wigner}.

\begin{lemma}
\label{lem:normscm}
In the setting of this section
$$
	\mathbf{E}\|X\| \lesssim
	\sqrt{\mathrm{Tr}[\Sigma]} + \sqrt{n\|\Sigma\|}.
$$
\end{lemma}

\begin{proof}
Note that
\begin{align*}
	\mathbf{E}(\langle v,Xw\rangle - \langle v',Xw'\rangle)^2
	&\le
	2 \,\mathbf{E}(\langle v-v',Xw\rangle)^2 +
	2 \,\mathbf{E}(\langle v',X(w-w')\rangle)^2 \\
	& =
	2\|\Sigma^{1/2}(v-v')\|^2\|w\|^2 +
	2\|\Sigma^{1/2}v'\|^2\|w-w'\|^2 \\
	&\le
	\mathbf{E}(X'_{v,w}-X'_{v',w'})^2
\end{align*}
when $\|v\|,\|w\|\le 1$, where we defined
$$
	X'_{v,w} = \sqrt{2}\,\langle v,\Sigma^{1/2}g\rangle +
	\sqrt{2}\,\|\Sigma\|^{1/2}\,\langle w,g'\rangle
$$
with $g,g'$ independent standard Gaussian vectors in $\mathbb{R}^d$
and $\mathbb{R}^n$, respectively.
Thus
$$
	\mathbf{E}\|X\| =
	\mathbf{E}\bigg[\sup_{v,w\in B}\langle v,Xw\rangle\bigg] \le
	\mathbf{E}\bigg[\sup_{v,w\in B}X'_{v,w}\bigg] \lesssim
	\mathbf{E}\|\Sigma^{1/2}g\| + \|\Sigma\|^{1/2}\mathbf{E}\|g\|
$$
by the Slepian-Fernique inequality.  The proof is easily completed.
\qed\end{proof}

The proof of the upper bound in Theorem \ref{thm:kolt} is now 
immediately completed by combining the results of Lemma 
\ref{lem:sqreduce} and Lemma \ref{lem:normscm}.

\begin{remark}

The proof of the upper bound given here reduces the problem of 
controlling the supremum of a Gaussian chaos process by decoupling to 
that of controlling the supremum of a Gaussian process. The original proof in 
\cite{KL16} uses a different method that exploits a much deeper general 
result on the suprema of empirical processes of squares, cf.\ 
\cite[Theorem 9.3.7]{Tal14}. While the route we have taken is much more 
elementary, the original approach has the advantage that it applies 
directly to subgaussian matrices.  The result of \cite{KL16} is also
stated for norms other than the spectral norm, but proof given here
extends readily to this setting.
\end{remark}

\subsubsection*{Lower bound}

It remains to prove the lower bound in Theorem \ref{thm:kolt}. The 
main idea behind the proof is that the decoupling inequality of
Lemma \ref{lem:decouple} can be partially reversed.

\begin{lemma}
\label{lem:recouple}
Let $\tilde X$ be an independent copy of $X$. Then for every
$v\in\mathbb{R}^d$
$$
	\mathbf{E}\|(Z-\Sigma)v\| \ge
	\frac{1}{n}\,\mathbf{E}\|X\tilde X^*v\|
	-
	\frac{\|\Sigma v\|}{\sqrt{n}}.
$$
\end{lemma}

\begin{proof}
The reader may readily verify that the random matrix
$$
	X' = \bigg(I-\frac{\Sigma vv^*}{\langle v,\Sigma v\rangle}\bigg) X
$$
is independent of the random vector $X^*v$ (and therefore of
$\langle v,Zv\rangle$). Moreover
$$
	(Z-\Sigma)v =
	\frac{XX^*v}{n} - \Sigma v =
	\frac{X'X^*v}{n} +
	\bigg(
	\frac{\langle v,Zv\rangle}{\langle v,\Sigma v\rangle} - 1\bigg)
	\Sigma v.
$$
As the columns of $X'$ are i.i.d.\ and independent of $X^*v$, the
pair $(X'X^*v,X^*v)$ has the same distribution as
$(X'_1\|X^*v\|,X^*v)$ where $X'_1$ denotes the first column of $X'$. Thus
$$
	\mathbf{E}\|(Z-\Sigma)v\| =
	\mathbf{E}\bigg\|
	\frac{X'_1\|X^*v\|}{n} +
	\bigg(
	\frac{\langle v,Zv\rangle}{\langle v,\Sigma v\rangle} - 1\bigg)
	\Sigma v
	\bigg\|
	\ge
	\frac{1}{n}\,
	\mathbf{E}\|X^*v\|\,\mathbf{E}\|X'_1\|,
$$
where we used Jensen's inequality conditionally on $X'$. 
Now note that
$$
	\mathbf{E}\|X_1'\| \ge
	\mathbf{E}\|X_1\| - \|\Sigma v\|\,
	\frac{\mathbf{E}|\langle v,X_1\rangle|}{\langle v,\Sigma v\rangle}
	\ge \mathbf{E}\|X_1\| -
        \frac{\|\Sigma v\|}{\langle v,\Sigma v\rangle^{1/2}}.
$$
We therefore have
$$
	\mathbf{E}\|(Z-\Sigma)v\| \ge
	\frac{1}{n}\,
	\mathbf{E}\|X_1\|\,\mathbf{E}\|\tilde X^*v\|
	-
	\frac{1}{n}\,
	\mathbf{E}\|X^*v\|
	\frac{\|\Sigma v\|}{\langle v,\Sigma v\rangle^{1/2}}
	\ge	
	\frac{1}{n}\,\mathbf{E}\|X\tilde X^*v\|
	-
	\frac{\|\Sigma v\|}{\sqrt{n}},
$$
as $\mathbf{E}\|X^*v\|\le \sqrt{n}\,\langle v,\Sigma v\rangle^{1/2}$ and
as $X_1\|\tilde X^*v\|$ has the same distribution as
$X\tilde X^*v$.
\qed\end{proof}

As a corollary, we can obtain the first term in the lower bound.

\begin{corollary}
\label{cor:firstterm}
In the setting of this section, we have
$$
	\mathbf{E}\|Z-\Sigma\| \gtrsim
	\|\Sigma\|\sqrt{\frac{r(\Sigma)}{n}}.
$$
\end{corollary}

\begin{proof}
Taking the supremum over $v\in B$ in Lemma \ref{lem:recouple} yields
$$
	\mathbf{E}\|Z-\Sigma\| 
	+
	\frac{\|\Sigma\|}{\sqrt{n}}
	\ge
	\sup_{v\in B}\frac{1}{n}\,\mathbf{E}\|X\tilde X^*v\| =
	\frac{1}{n}\,\mathbf{E}\|X_1\|\,
	\sup_{v\in B}\mathbf{E}\|\tilde X^*v\|.
$$
Using Gaussian concentration as in the proof of Lemma \ref{lem:lowernck},
we obtain
$$
	\mathbf{E}\|X_1\| \gtrsim
	\mathbf{E}[\|X_1\|^2]^{1/2} = \sqrt{\mathrm{Tr}[\Sigma]},\qquad
	\mathbf{E}\|\tilde X^*v\| \gtrsim
	\mathbf{E}[\|\tilde X^*v\|^2]^{1/2} =
	\sqrt{n\,\langle v,\Sigma v\rangle}.
$$
This yields
$$
	\mathbf{E}\|Z-\Sigma\| 
	+
	\frac{\|\Sigma\|}{\sqrt{n}}
	\gtrsim
	\|\Sigma\|\sqrt{\frac{r(\Sigma)}{n}}.
$$
On the other hand, we can estimate
by the central limit theorem
$$
	\frac{\|\Sigma\|}{\sqrt{n}} \lesssim
	\sup_{v\in B}\mathbf{E}|\langle v,(Z-\Sigma)v\rangle| \le
	\mathbf{E}\|Z-\Sigma\|,
$$
as $\langle v,(Z-\Sigma)v\rangle =
\langle v,\Sigma v\rangle\, \frac{1}{n}\sum_{k=1}^n\{Y_k^2-1\}$ with
$Y_k=\langle v,X_k\rangle/\langle v,\Sigma v\rangle^{1/2}\sim
N(0,1)$.
\qed\end{proof}

We can now easily complete the proof of Theorem \ref{thm:kolt}.

\begin{proof}[Proof of Theorem \ref{thm:kolt}]
The upper bound follows immediately from Lemmas
\ref{lem:sqreduce} and \ref{lem:normscm}.
For the lower bound, suppose first that $r(\Sigma)\le 2n$. Then
$\sqrt{r(\Sigma)/n}\gtrsim r(\Sigma)/n$, and the result follows from
Corollary \ref{cor:firstterm}.  On the other hand, if
$r(\Sigma)>2n$,
$$
	\mathbf{E}\|Z-\Sigma\| \ge
	\mathbf{E}\|Z\| - \|\Sigma\|
	\ge
	\frac{\mathbf{E}\|X_1\|^2}{n} - \|\Sigma\|\,\frac{r(\Sigma)}{2n}  =
	\|\Sigma\|\,\frac{r(\Sigma)}{2n},
$$
where we used that $Z = \frac{1}{n}\sum_{k=1}^nX_kX_k^*\succeq 
\frac{1}{n}X_1X_1^*$.
\qed\end{proof}

\begin{acknowledgement} The author warmly thanks IMA for its hospitality 
during the annual program ``Discrete Structures: Analysis and 
Applications'' in Spring 2015. The author also thanks Markus Rei{\ss} for 
the invitation to lecture on this material in the 2016 spring school in 
L\"ubeck, Germany, which further motivated the exposition in this chapter. 
This work was supported in part by NSF grant CAREER-DMS-1148711 and by ARO 
PECASE award W911NF-14-1-0094. \end{acknowledgement}

%


\begin{thebibliography}{10}
\providecommand{\url}[1]{{#1}}
\providecommand{\urlprefix}{URL }
\expandafter\ifx\csname urlstyle\endcsname\relax
  \providecommand{\doi}[1]{DOI~\discretionary{}{}{}#1}\else
  \providecommand{\doi}{DOI~\discretionary{}{}{}\begingroup
  \urlstyle{rm}\Url}\fi

\bibitem{AKM13}
Agarwal, N., Kolla, A., Madan, V.: Small lifts of expander graphs are expanding
  (2013).
\newblock Preprint arXiv:1311.3268

\bibitem{AGZ10}
Anderson, G.W., Guionnet, A., Zeitouni, O.: An introduction to random matrices,
  \emph{Cambridge Studies in Advanced Mathematics}, vol. 118.
\newblock Cambridge University Press, Cambridge (2010)

\bibitem{BS98}
Bai, Z.D., Silverstein, J.W.: No eigenvalues outside the support of the
  limiting spectral distribution of large-dimensional sample covariance
  matrices.
\newblock Ann. Probab. \textbf{26}(1), 316--345 (1998)

\bibitem{BvH16}
Bandeira, A.S., {Van Handel}, R.: Sharp nonasymptotic bounds on the norm of
  random matrices with independent entries.
\newblock Ann. Probab.  (2016).
\newblock To appear

\bibitem{BLM13}
Boucheron, S., Lugosi, G., Massart, P.: Concentration inequalities.
\newblock Oxford University Press, Oxford (2013)

\bibitem{Che78}
Chevet, S.: S\'eries de variables al\'eatoires gaussiennes \`a valeurs dans
  {$E\hat \otimes _{\varepsilon }F$}. {A}pplication aux produits d'espaces de
  {W}iener abstraits.
\newblock In: S\'eminaire sur la {G}\'eom\'etrie des {E}spaces de {B}anach
  (1977--1978), pp. Exp. No. 19, 15. \'Ecole Polytech., Palaiseau (1978)

\bibitem{Dir15}
Dirksen, S.: Tail bounds via generic chaining.
\newblock Electron. J. Probab. \textbf{20}, no. 53, 29 (2015)

\bibitem{EY12}
Erd{\H{o}}s, L., Yau, H.T.: Universality of local spectral statistics of random
  matrices.
\newblock Bull. Amer. Math. Soc. (N.S.) \textbf{49}(3), 377--414 (2012)

\bibitem{Gor85}
Gordon, Y.: Some inequalities for {G}aussian processes and applications.
\newblock Israel J. Math. \textbf{50}(4), 265--289 (1985)

\bibitem{KL16}
Koltchinskii, V., Lounici, K.: Concentration inequalities and moment bounds for
  sample covariance operators.
\newblock Bernoulli  (2016).
\newblock To appear

\bibitem{Lat05}
Lata{\l}a, R.: Some estimates of norms of random matrices.
\newblock Proc. Amer. Math. Soc. \textbf{133}(5), 1273--1282 (electronic)
  (2005)

\bibitem{LT91}
Ledoux, M., Talagrand, M.: Probability in {B}anach spaces, \emph{Ergebnisse der
  Mathematik und ihrer Grenzgebiete}, vol.~23.
\newblock Springer-Verlag, Berlin (1991).
\newblock Isoperimetry and processes

\bibitem{LP86}
Lust-Piquard, F.: In\'egalit\'es de {K}hintchine dans {$C_p\;(1<p<\infty)$}.
\newblock C. R. Acad. Sci. Paris S\'er. I Math. \textbf{303}(7), 289--292
  (1986)

\bibitem{MJCFT14}
Mackey, L., Jordan, M.I., Chen, R.Y., Farrell, B., Tropp, J.A.: Matrix
  concentration inequalities via the method of exchangeable pairs.
\newblock Ann. Probab. \textbf{42}(3), 906--945 (2014)

\bibitem{MSS15}
Marcus, A.W., Spielman, D.A., Srivastava, N.: Interlacing families {II}:
  {M}ixed characteristic polynomials and the {K}adison-{S}inger problem.
\newblock Ann. of Math. (2) \textbf{182}(1), 327--350 (2015)

\bibitem{Pis86}
Pisier, G.: Probabilistic methods in the geometry of {B}anach spaces.
\newblock In: Probability and analysis ({V}arenna, 1985), \emph{Lecture Notes
  in Math.}, vol. 1206, pp. 167--241. Springer, Berlin (1986)

\bibitem{Pis03}
Pisier, G.: Introduction to operator space theory, \emph{London Mathematical
  Society Lecture Note Series}, vol. 294.
\newblock Cambridge University Press, Cambridge (2003)

\bibitem{RS75}
Reed, M., Simon, B.: Methods of modern mathematical physics. {II}. {F}ourier
  analysis, self-adjointness.
\newblock Academic Press, New York-London (1975)

\bibitem{RS13}
Riemer, S., Sch{\"u}tt, C.: On the expectation of the norm of random matrices
  with non-identically distributed entries.
\newblock Electron. J. Probab. \textbf{18}, no. 29, 13 (2013)

\bibitem{Rud99}
Rudelson, M.: Almost orthogonal submatrices of an orthogonal matrix.
\newblock Israel J. Math. \textbf{111}, 143--155 (1999)

\bibitem{Seg00}
Seginer, Y.: The expected norm of random matrices.
\newblock Combin. Probab. Comput. \textbf{9}(2), 149--166 (2000)

\bibitem{Sod10}
Sodin, S.: The spectral edge of some random band matrices.
\newblock Ann. of Math. (2) \textbf{172}(3), 2223--2251 (2010)

\bibitem{SV13}
Srivastava, N., Vershynin, R.: Covariance estimation for distributions with
  {$2+\varepsilon$} moments.
\newblock Ann. Probab. \textbf{41}(5), 3081--3111 (2013)

\bibitem{Tal14}
Talagrand, M.: Upper and lower bounds for stochastic processes, vol.~60.
\newblock Springer, Heidelberg (2014)

\bibitem{Tao12}
Tao, T.: Topics in random matrix theory, \emph{Graduate Studies in
  Mathematics}, vol. 132.
\newblock American Mathematical Society, Providence, RI (2012)

\bibitem{TV14}
Tao, T., Vu, V.: Random matrices: the universality phenomenon for {W}igner
  ensembles.
\newblock In: Modern aspects of random matrix theory, \emph{Proc. Sympos. Appl.
  Math.}, vol.~72, pp. 121--172. Amer. Math. Soc., Providence, RI (2014)

\bibitem{TJ74}
Tomczak-Jaegermann, N.: The moduli of smoothness and convexity and the
  {R}ademacher averages of trace classes {$S_{p}(1\leq p<\infty )$}.
\newblock Studia Math. \textbf{50}, 163--182 (1974)

\bibitem{Tro15}
Tropp, J.: An introduction to matrix concentration inequalities.
\newblock Foundations and Trends in Machine Learning  (2015)

\bibitem{Tro15b}
Tropp, J.: Second-order matrix concentration inequalities (2015).
\newblock Preprint arXiv:1504.05919

\bibitem{vH16}
{Van Handel}, R.: Chaining, interpolation, and convexity.
\newblock J. Eur. Math. Soc.  (2016).
\newblock To appear

\bibitem{vH15}
{Van Handel}, R.: On the spectral norm of {G}aussian random matrices.
\newblock Trans. Amer. Math. Soc.  (2016).
\newblock To appear

\bibitem{Ver12}
Vershynin, R.: Introduction to the non-asymptotic analysis of random matrices.
\newblock In: Compressed sensing, pp. 210--268. Cambridge Univ. Press,
  Cambridge (2012)

\end{thebibliography}

\end{document}